\documentclass[11pt,reqno]{amsart}
\usepackage{color, amsmath,amssymb, amsfonts, amstext,amsthm, latexsym}
\allowdisplaybreaks
\newcommand{\cF}{{\cal F}}
\newcommand{\RR}{{\mathbb R}}
\newcommand{\HH}{{\mathcal H}}
\newcommand{\EX}{{\mathbb E}}
\newcommand{\EE}{{\mathbb E}}

  \newcommand{\PX}{{\mathbb P}}
\newcommand{\PP}{{\mathbb P}}

\newcommand{\s}{\sigma}

\newcommand{\om}{\omega}
\newcommand{\Om}{\Omega}

\newcommand{\de}{\delta}

\renewcommand{\cF}{\mathcal F}

\numberwithin{equation}{section}
\newtheorem{theorem}{Theorem}[section]
\newtheorem{defn}[theorem]{Definition}

\newtheorem{lemma}[theorem]{Lemma}
\newtheorem{remark}[theorem]{Remark}
\newtheorem{prop}[theorem]{Proposition}

\begin{document}
\title[ LDP  and inviscid shell models]
{Large deviation principle and inviscid shell models}

\author[H. Bessaih ] {Hakima Bessaih }
\address{
H. Bessaih,  University of Wyoming,
Department of Mathematics,
Dept. 3036,
1000 East University Avenue,
Laramie WY 82071,
United States}
\email{Bessaih@uwyo.edu}

 \author[A. Millet]{Annie Millet }
\address{A.  Millet,  SAMOS, Centre d'\'Economie de la Sorbonne,
 Universit\'e Paris 1 Panth\'eon Sorbonne,
90 Rue de Tolbiac, 75634 Paris Cedex France {\it and}
Laboratoire de Probabilit\'es et Mod\`eles Al\'eatoires,
  Universit\'es Paris~6-Paris~7, Bo\^{\i}te Courrier 188,
      4 place Jussieu, 75252 Paris Cedex 05, France}
\email{  annie.millet@univ-paris1.fr {\it and}
 annie.millet@upmc.fr}



\begin{abstract}
A LDP is proved for the inviscid shell model of turbulence. As the
viscosity coefficient $\nu$ converges to 0 and the noise intensity
is multiplied by $\sqrt{\nu}$, we prove that some shell models of
turbulence with a multiplicative stochastic perturbation driven by a
$H$-valued Brownian motion satisfy a LDP in ${\mathcal C}([0,T],
V)$ for the topology of uniform convergence on $[0,T]$, but
where $V$ is endowed with a  topology weaker than the natural one.
 The initial condition has to belong to $V$
and the proof is based on the weak convergence of a family of
stochastic control equations. The rate function is described in
terms of the solution to  the inviscid equation.
\end{abstract}

\maketitle






\section{Introduction}\label{s1}

Shell models,  from E.B. Gledzer, K. Ohkitani, M. Yamada, are
simplified Fourier systems with respect to the Navier-Stokes ones,
where the interaction between different modes is preserved only
between nearest neighbors. These are some of the most interesting
examples of artificial models of fluid dynamics that capture some
properties of turbulent fluids like power law decays of structure
functions.

There is an extended literature on shell models. We refer to
K.~Ohkitani and  M.~ Yamada \cite{OY89},
V.~ S.~Lvov, E.~Podivilov, A.~Pomyalov, I.~Procaccia and D.~Vandembroucq \cite{LPPPV98},
 L.~Biferale  \cite{Biferale} and the references
therein. However, these papers are  mainly dedicated to the numerical approach
and pertain  to the finite dimensional case. In a recent work
by P.~Constantin, B.~Levant and  E.~S.~Titi \cite{CLT06},
 some results of regularity, attractors and inertial
manifolds are proved for deterministic infinite dimensional shells
models. In \cite{CLT07-inviscid} these authors have proved  some regularity results
for the inviscid case. The infinite-dimensional stochastic
version of shell models have been studied by D. Barbato, M.
Barsanti, H. Bessaih and F. Flandoli in \cite{BBBF06} in the case of an
additive random perturbation. Well-posedeness and apriori estimates
were obtained, as well as the existence of an invariant measure.
Some balance laws have been investigated and preliminary results
about the structure functions have been presented.

The more general formulation involving a multiplicative noise reads
as follows
 \[ du(t) + [\nu A u(t) + B(u(t),u(t))]\, dt =  \s(t,u(t))\, dW_t\, , \quad u(0)=\xi.\]
driven by a Hilbert space-valued Brownian motion $W$. It  involves
some similar bilinear operator $B$ with antisymmetric properties and
some linear "second order"  (Laplace) operator $A$ which is
regularizing and multiplied by some non negative coefficient $\nu$
which stands for the viscosity in the usual hydro-dynamical models.
The shell models  are adimensional and the bilinear term is better behaved than
that in the Navier Stokes equation. Existence, uniqueness and
several properties were studied in \cite{BBBF06} in the case on an
additive noise and in \cite{CM} for a multiplicative  noise in the
"regular" case of a non-zero viscosity coefficient which was taken
constant.

Several recent papers have studied a Large Deviation Principle (LDP) for the distribution
 of the solution
to a hydro-dynamical stochastic  evolution equation: S.~Sritharan and P.~Sundar \cite{Sundar} for the
2D Navier Stokes equation, J.~Duan and A.~Millet \cite{DM} for the Boussinesq model,
 where the Navier Stokes
equation is coupled with a similar nonlinear equation describing the temperature evolution,
  U.~Manna, S.~Sritharan  and P.~Sundar \cite{MSS} for shell models of turbulence,
 I.~Chueshov and A.~Millet
\cite{CM} for a wide class of hydro-dynamical equations including
the 2D B\'enard magneto-hydro dynamical
and 3D $\alpha$-Leray Navier Stokes models,
 A.Du, J. Duan and H. Gao \cite{DDG} for two layer quasi-geostrophic
flows modeled by coupled equations with a bi-Laplacian.
All the above papers consider an equation with a given
(fixed) positive viscosity coefficient and study exponential concentration to a deterministic
model when the noise intensity is multiplied by a coefficient $\sqrt{\epsilon}$ which converges to 0.
All these  papers deal with a multiplicative noise and
use the weak convergence approach of LDP, based on the Laplace principle, developed by
P.  Dupuis and R. Ellis in \cite{DupEl97}.
This approach has shown to be successful in several other
infinite-dimensional
cases (see e.g. \cite{BD00}, \cite{BD07},  \cite{WeiLiu})
 and differ from that used to get LDP in finer
topologies for quasi-linear SPDEs, such as \cite{Sowers}, \cite{ChenalMillet},
 \cite{CR}, \cite{Chang}.
For hydro-dynamical models, the LDP was proven in the natural space of trajectories,
that is ${\mathcal C}([0,T],H)
\cap L^2([0,T],V)$, where roughly speaking,  $H$ is $L^2$ and $V= Dom(A^{\frac{1}{2}})$
is the Sobolev space $H_1^2$
with proper periodicity or boundary conditions. The initial condition $\xi$  only
 belongs  to $H$.

The aim of this paper is different. Indeed, the asymptotics we are interested in have a physical
meaning, namely the viscosity coefficient $\nu$ converges to 0. Thus the limit equation, which
corresponds to the inviscid case, is much more difficult to deal with, since the regularizing effect
of the  operator $A$ does not help anymore. Thus, in order to get existence, uniqueness and
apriori estimates to the inviscid equation, we need to start from some more regular initial condition
$\xi\in V$, to impose that $(B(u,u),Au)=0$ for all $u$ regular enough
 (this identity  would be true in the case on the 2D Navier Stokes
equation under proper periodicity properties); note that this equation  is satisfied in the GOY and Sabra  shell
models of turbulence under a suitable relation on the coefficients $a,b$ and $\mu$ stated below.
 Furthermore,  some more conditions on the diffusion coefficient
are required as well. The intensity of the noise has to be
multiplied by $\sqrt{\nu}$ for the convergence to hold.

 The
technique is again that of the weak convergence.  One proves that
given a family  $(h_\nu)$ of random elements of the RKHS of $W$
which converges weakly  to $h$,  the corresponding family of
stochastic control equations,  deduced from the original ones by
shifting the noise by $\frac{h_\nu}{\sqrt {\nu}}$, converges in
distribution to the limit inviscid equation where the Gaussian
noise $W$ has been replaced by $h$. Some apriori control of the
solution to such equations has to be proven uniformly in $\nu>0$
for "small enough" $\nu$.
 Existence and uniqueness as well as
apriori bounds have to be obtained  for the inviscid limit
equation. Some upper bounds of time increments have to be proven
for the inviscid equation and the stochastic model with a small
viscosity coefficient; they are similar to that  in \cite{DM} and
\cite{CM}. The LDP can be shown in ${\mathcal C}([0,T],V)$ for the topology
of uniform convergence on $[0,T]$, but where $V$ is endowed with a weaker
topology, namely that induced by the $H$ norm.
More generally, under some slight extra assumption on the
diffusion coefficient $\s$, the LDP is proved in ${\mathcal C}([0,T],V)$ where
$V$ is endowed with the norm $\|\cdot\|_\alpha:=|A^\alpha (\cdot) |_H$ for
$0\leq \alpha\leq \frac{1}{4}$. The natural case
$\alpha=\frac{1}{2}$ is out of reach because the inviscid limit
equation is much more irregular. Indeed, it is an abstract
equivalent of the Euler equation.
The case $\alpha=0$  corresponds to $H$ and then no more condition on $\s$ is required.
The case  $\alpha = \frac{1}{4}$ is that of an interpolation space which plays a crucial
role in the 2D Navier Stokes equation.
 Note that in  the  different context of a scalar equation, M. Mariani \cite{MaMa}  has also proved  a
LDP for a stochastic PDE when a coefficient $\varepsilon$
 in front of a deterministic operator converges to 0
and the intensity of the Gaussian noise is multiplied by $\sqrt{\varepsilon}$.
However, the physical model and the technique used in   \cite{MaMa} are completely different from ours.
\smallskip

The paper is organized as follows.
 Section 2 gives a precise description of the model and proves apriori bounds
for the norms in  ${\mathcal C}([0,T],H)$ and $  L^2([0,T],V)$ of the stochastic control
 equations uniformly
in the viscosity coefficient $\nu\in ]0,\nu_0]$ for small enough $\nu_0$. Section 3 is
 mainly devoted to
prove existence, uniqueness of the solution to the deterministic inviscid
 equation with an external multiplicative
impulse driven by an element of the RKHS of $W$, as well as apriori bounds
of the solution in ${\mathcal C}([0,T],
V)$ when the initial condition belong to $V$ and under reinforced assumptions on $\s$. Under these extra
assumptions, we are able to improve the apriori estimates of the solution and establish them in
 ${\mathcal C}([0,T],V)$ and $  L^2([0,T],Dom(A))$.
 Finally the weak convergence and compactness of the level sets
of the rate function are proven in section 4; they  imply the LDP in ${\mathcal C}([0,T],V)$
where $V$ is endowed with  the weaker norm associated with $A^\alpha$
 for any value of $\alpha$ with $0\leq \alpha \leq \frac{1}{4}$.

The  LDP for the 2D Navier Stokes equation as the viscosity
coefficient converges to 0 will be studied in a forthcoming paper.

We will denote by $C$ a constant which may change from one line to the next,
 and $C(M)$ a constant depending on $M$.
\section{Description of the model} \label{s2}
\subsection{GOY and Sabra shell models}  Let $H$ be the set of all sequences
 $u=(u_1, u_2,\ldots)$ of complex numbers
such that $\sum_n |u_n|^2<\infty$. We consider $H$ as a \emph{real} Hilbert space
endowed  with the inner product $(\cdot,\cdot)$ and the norm $|\cdot|$ of the form
\begin{equation}\label{normH}
(u,v)={\rm Re}\,\sum_{n\geq 1}u_n v_n^*,\quad
|u|^2 =\sum_{n\geq 1} |u_n|^2,
\end{equation}
where $v_n^*$ denotes the complex conjugate of $v_n$.
Let $k_0>0$,  $\mu>1$ and for every $n\geq 1$, set $k_n=k_0\, \mu^n$.
Let $A:Dom(A)\subset H \to H $ be the non-bounded linear operator defined by
\[
(Au)_n = k_n^2 u_n,\quad n=1,2,\ldots,\qquad Dom(A)=\Big\{ u\in H\,
:\; \sum_{n\geq 1} k_n^4 |u_n|^2<\infty\Big\}.
\]
The operator $A$ is clearly self-adjoint, strictly positive definite since $(Au,u)\geq k_0^2 |u|^2$
for $u\in Dom(A)$.
For any $\alpha >0$, set
\begin{equation} \label{Halpha}
{\mathcal H}_\alpha = Dom(A^\alpha) = \{ u\in H \, :\, \sum_{n\geq 1} k_n^{4\alpha} |u_n|^2 <+\infty\},\;
\|u\|^2_\alpha = \sum_{n\geq 1} k_n^{4\alpha} |u_n|^2 \; \mbox{\rm for }\; u\in {\mathcal H}_\alpha.
\end{equation}
Let  ${\mathcal H}_0=H$,
\[ V:=Dom(A^{\frac{1}{2}}) =
 \Big\{ u\in H \, :\, \sum_{n\geq 1} k_n^2 |u_n|^2 <+\infty\Big\}\, ;\; \mbox{\rm also set  }\;
{\mathcal H}= {\mathcal H}_{\frac{1}{4}},\, \|u\|_{\mathcal H} = \|u\|_{\frac{1}{4}}. \]
Then $V$ (as each of the spaces ${\mathcal H}_\alpha$)  is a Hilbert space
for the scalar product $(u,v)_V = Re(\sum_n k_n^2\, u_n\, v_n^*)$, $u,v\in V$
and the associated norm is denoted by
\begin{equation} \label{normV}
\|u\|^2 = \sum_{n\geq 1} k_n^2\, |u_n|^2.
\end{equation}
The adjoint of $V$ with respect to  the $H$ scalar product
is $V' = \{ (u_n)\in {\mathbb C}^{\mathbb N} \, :\,
\sum_{n\geq 1} k_n^{-2}\, |u_n|^2 <+\infty\}$ and $V\subset H\subset V'$
is a Gelfand triple. Let $\langle u\, ,\, v\rangle = Re\left (\sum_{n\geq 1} u_n\, v_n^*\right)$ denote
the duality between $u\in V$ and $v\in V'$.
Clearly for $0\leq \alpha <\beta$,  $u\in {\mathcal H^\beta}$ and $v\in V$ we have
\begin{equation}\label{compcalH}
\|u\|_\alpha^2 \leq k_0^{4(\alpha - \beta)}\, \|u\|^2_\beta \, ,
\; \mbox{\rm and }\;  \|v\|^2_{\mathcal H} \leq |v|\,  \|v\|,
\end{equation}
where the last inequality is proved by the Cauchy-Schwarz inequality.

Set $ u_{-1}= u_{0}=0 $, let $a,b$ be real numbers and
 $B : H\times V \to H$ (or  $B : V\times H \to H$) denote the bilinear operator  defined by
\begin{equation} \label{GOY}
\left[B(u,v)\right]_n=-i\left( a k_{n+1} u_{n+1}^* v_{n+2}^*
+b k_{n} u_{n-1}^* v_{n+1}^* -a k_{n-1} u_{n-1}^* v_{n-2}^*
-b k_{n-1} u_{n-2}^* v_{n-1}^*
\right)
\end{equation}
for $n=1,2,\ldots$ in the  GOY shell-model (see, e.g., \cite{OY89})
or
\begin{equation} \label{Sabra}
\left[B(u,v)\right]_n=-i\left( a k_{n+1} u_{n+1}^*\,  v_{n+2}
+b k_{n} u_{n-1}^* v_{n+1} +a k_{n-1} u_{n-1} v_{n-2}
+b k_{n-1} u_{n-2} v_{n-1}
\right),
\end{equation}
in the  Sabra shell model introduced in \cite{LPPPV98}.

Note that $B$ can be extended as a bilinear operator from $H\times H$ to $V'$ and that  there exists
a constant ${C}>0$ such that  given
$u,v\in H$ and $w\in V$ we have
\begin{equation}\label{trili}
|\langle B(u,v)\, ,\, w\rangle | + |\big( B(u,w)\, ,\, v\big) | + |\big( B(w,u)\, ,\, v\big) |
\leq {C}\, |u|\, |v|\, \|w\|.
\end{equation}
An easy computation proves that for $u,v\in H$ and $w\in V$ (resp. $v,w\in H$ and $u\in V$),
\begin{equation}\label{antisym}
\langle B(u,v)\, ,\, w\rangle = - \big(B(u,w)\, ,\, v\big) \; \mbox{\rm (resp.  }\,
\big( B(u,v)\, ,\, w\big)  = - \big(B(u,w)\, ,\, v\big) \mbox{\rm  \;)}.
\end{equation}
Furthermore, $B:V\times V\to V$ and $B : {\mathcal H}\times {\mathcal H} \to H$;
 indeed, for $u,v\in V$  (resp.  $u,v\in {\mathcal H}$)
 we have
\begin{align}\label{boundB1}
 \|B(u,v)\|^2 & = \sum_{n\geq 1} k_n^2\, |B(u,v)_n|^2\, \leq C\, \|u\|^2 \sup_n k_n^2 |v_n|^2
\leq C\, \|u\|^2\, \|v\|^2,\\
| B(u,v)| & \leq  C\, \|u\|_{\mathcal H} \, \|v\|_{\mathcal H}.
\nonumber
\end{align}
For $u,v$ in either $H$, ${\mathcal H}$ or $V$, let $B(u):=B(u,u)$.
The anti-symmetry property  \eqref{antisym} implies  that
$ |\langle B(u_1)-B(u_2)\, ,\, u_1-u_2\rangle_V| =
 |\langle B(u_1-u_2), u_2\rangle_V|$ for $u_1,u_2\in V$ and
 $ |\langle B(u_1)-B(u_2)\, ,\, u_1-u_2\rangle| =
 |\langle B(u_1-u_2), u_2\rangle|$ for $u_1\in H$ and $u_2\in V$. Hence there exist  positive constants
 $\bar{C}_1$ and $\bar{C}_2$ such that
\begin{eqnarray}\label{diffBV}
|\langle B(u_1)-B(u_2)\, ,\, u_1-u_2\rangle_V| &  \leq  &
\bar{C}_1 \,  \|u_1-u_2\|^2\,   \|u_2\|, \forall u_1,u_2\in V, \\
|\langle B(u_1)-B(u_2)\, ,\, u_1-u_2\rangle| & \leq &
\bar{C}_2 \,  |u_1-u_2|^2\,   \|u_2\|, \forall u_1\in H, \forall u_2\in V \label{diffB-HV}.
\end{eqnarray}
Finally, since $B$ is bilinear, Cauchy-Schwarz's inequality yields
for any $\alpha \in [ 0,\frac{1}{2}]$, $u,v\in V$:
\begin{align}\label{AalphaB}
\big|\big( A^\alpha B(u)- A^\alpha B(v) \, ,\,  A^\alpha(u-v)\big) \big|&
 \leq \big| \big( A^\alpha B(u-v,u)
+  A^\alpha B(v,u-v)\, ,\, A^\alpha (u-v)\big)\big|
\nonumber \\
&  \leq C \|u-v\|^2_\alpha
\, (\|u\|+\|v\|).
\end{align}

In the  GOY shell model,  $B$ is defined by \eqref{GOY};  for any $u\in V$, $Au\in V'$  we have
\[ \langle B(u,u), A u\rangle = Re\Big( -i \sum_{n\geq 1} u_n^*\, u_{n+1}^* \, u_{n+2}^* \mu^{3n+1}\Big)
\, k_0^3 (a+b\mu^2 -a \mu^4 -b\mu^4).\]
Since $\mu \neq 1$,
\begin{equation}\label{B(u)AuG}
 a(1+\mu^2) +b \mu^2 =0\quad \mbox{\rm if and only if  } \; \langle B(u,u)\, ,\, Au\rangle =0, \forall u\in V.
\end{equation}
On the other hand, in the Sabra shell model,  $B$ is defined by \eqref{Sabra} and  one has for
$u\in V$,
\[ \langle B(u,u)\, ,\, Au\rangle = k_0^3 Re \Big( -i \sum_{n\geq 1} \mu^{3n+1}\, \Big[ (a+b\, \mu^2) \,
u_n^*\, u_{n+1}^*\, u_{n+2} + (a+b) \mu^4 u_n \, u_{n+1}\,
u_{n+2}^*\Big]\Big).\] Thus   $(B(u,u),Au)=0$ for every $u\in V$
if and only if $a+b\mu^2=(a+b)\mu^4$ and again  $\mu \neq 1$ shows
that \eqref{B(u)AuG} holds true.

\subsection{Stochastic driving force}
Let $Q$ be a linear
positive  operator in the Hilbert space $H$ which is  trace class,
and hence   compact. Let $H_0 = Q^{\frac12} H$; then $H_0$ is a
Hilbert space with the scalar product
$$
(\phi, \psi)_0 = (Q^{-\frac12}\phi, Q^{-\frac12}\psi),\; \forall
\phi, \psi \in H_0,
$$
together with the induced norm $|\cdot|_0=\sqrt{(\cdot,
\cdot)_0}$. The embedding $i: H_0 \to  H$ is Hilbert-Schmidt and
hence compact, and moreover, $i \; i^* =Q$.
Let $L_Q\equiv L_Q(H_0,H) $ be the space of linear operators $S:H_0\mapsto H$ such that
$SQ^{\frac12}$ is a Hilbert-Schmidt operator  from $H$ to $H$. The norm in the space $L_Q$ is
  defined by  $|S|_{L_Q}^2 =tr (SQS^*)$,  where $S^*$ is the adjoint operator of
$S$. The $L_Q$-norm can be also written in the form
\begin{equation}\label{LQ-norm}
 |S|_{L_Q}^2=tr ([SQ^{1/2}][SQ^{1/2}]^*)=\sum_{k\geq 1} |SQ^{1/2}\psi_k|^2=
 \sum_{k\geq 1} |[SQ^{1/2}]^*\psi_k|^2
\end{equation}
for any orthonormal basis  $\{\psi_k\}$ in $H$, for example $(\psi_k)_n = \delta_n^k$.
\par
Let   $W(t)$ be a   Wiener process  defined   on a filtered
probability space $(\Om, \cF, (\cF_t), \PX)$, taking values in $H$
and with covariance operator $Q$. This means that $W$ is Gaussian, has independent
time increments and that for $s,t\geq 0$, $f,g\in H$,
\[
\EE  (W(s),f)=0\quad\mbox{and}\quad
\EE  (W(s),f) (W(t),g) = \big(s\wedge t)\, (Qf,g).
\]
Let  $\beta_j$ be  standard (scalar) mutually independent Wiener processes,
$\{ e_j\}$ be an  orthonormal basis in $H$ consisting of eigen-elements of $Q$, with
$Qe_j=q_je_j$. Then $W$ has  the following  representation
\begin{equation}\label{W-n}
W(t)=\lim_{n\to\infty} W_n(t)\;\mbox{ in }\; L^2(\Om; H)\; \mbox{ with }
W_n(t)=\sum_{1\leq j\leq n} q^{1/2}_j \beta_j(t) e_j,
\end{equation}
and $Trace(Q)=\sum_{j\geq 1} q_j$. For details concerning this Wiener process
see e.g.  \cite{PZ92}.
\par
Given a viscosity coefficient $\nu >0$, consider the following stochastic shell model
\begin{equation} \label{Sshell}
d_t u(t)+ \big[ \nu A u(t) + B(u(t))\big]\, dt = \sqrt{\nu}\,  \sigma_\nu(t,u(t))\, dW(t),
\end{equation}
where
the noise intensity $\s_\nu: [0, T]\times V \to L_Q(H_0, H)$ of the stochastic perturbation
is properly normalized by the square root of the viscosity coefficient $\nu$.
We assume that $\sigma_\nu$ satisfies the following growth and Lipschitz conditions:
\par
\noindent \textbf{Condition (C1):} {\it
$\s_\nu \in {\mathcal C}\big([0, T] \times V; L_Q(H_0, H)\big)$,
 and there exist non negative  constants $K_i$ and $L_i$
such that for every $t\in [0,T]$ and $u,v\in V$:\\
{\bf (i)}   $|\s_\nu(t,u)|^2_{L_Q} \leq K_0+ K_1 |u|^2+ K_2 \|u\|^2$, \\
{\bf (ii)}   $|\s_\nu(t,u)-\s_\nu(t,v)|^2_{L_Q}
\leq L_1 |u-v|^2 + L_2 \|u-v\|^2$.
}
\smallskip

For technical reasons, in order to prove a large deviation principle for the
distribution of  the solution to \eqref{Sshell} as the viscosity coefficient $\nu$ converges to 0,
 we will need some precise estimates
on the solution of the equation  deduced from \eqref{Sshell} by  shifting
 the Brownian $W$ by some
random element  of its RKHS.  This cannot be deduced from similar ones  on $u$ by means
of a Girsanov transformation since the Girsanov density is not
uniformly bounded when the intensity of the noise tends to zero
(see e.g. \cite{DM} or \cite{CM}).
\par
To describe a set of admissible random shifts,   we introduce the class
 $\mathcal{A}$ as the  set of $H_0-$valued
$(\cF_t)-$predictable stochastic processes $h$ such that
$\int_0^T |h(s)|^2_0 ds < \infty, \; $ a.s.
For fixed $M>0$, let
\[S_M=\Big\{h \in L^2(0, T; H_0): \int_0^T |h(s)|^2_0 ds \leq M\Big\}.\]
The set $S_M$, endowed with the following weak topology, is a
  Polish  (complete separable metric)  space
(see e.g. \cite{BD07}):
$ d_1(h, k)=\sum_{k=1}^{\infty} \frac1{2^k} \big|\int_0^T \big(h(s)-k(s),
\tilde{e}_k(s)\big)_0 ds \big|,$
where $
\{\tilde{e}_k(s)\}_{k=1}^{\infty}$ is an  orthonormal basis
for $L^2([0, T], H_0)$.
For $M>0$ set
\begin{equation} \label{AM}
 \mathcal{A}_M=\{h\in \mathcal{A}: h(\om) \in
 S_M, \; a.s.\}.
\end{equation}
In order to define the stochastic control equation, we introduce for $\nu\geq 0$ a family of   intensity
coefficients $\tilde{\s}_\nu$ of a  random element $h\in {\mathcal A}_M$ for some $M>0$. The case
$\nu=0$ will be that of an inviscid limit
"deterministic" equation with no stochastic integral and which can be dealt with for fixed $\omega$.
We assume that for any $\nu\geq 0$ the coefficient  $\tilde{\sigma}_\nu$
satisfies the following condition:
\medskip\par
\noindent \textbf{Condition (C2):}    {\it ${}\;{\tilde \s}_\nu
\in {\mathcal C}\big([0, T] \times V; L(H_0, H)\big)$ and there exist constants
 $\tilde{K}_{\mathcal H}$,   $\tilde{K}_i$,  and $\tilde{L}_j$, for
$i=0,1$ and $j=1,2$ such that:
\begin{align} |\tilde{\s}_\nu (t,u)|^2_{L(H_0,H)} \leq \tilde{K}_0 + \tilde{K}_1 |u|^2 +
\nu \tilde{K}_\HH \|u\|_{\mathcal H}^2, & \quad \forall t\in [0,T], \;
 \forall u\in V, \label{tilde-s-b}\\
 |\tilde{\s}_\nu(t,u) -\tilde{\s}_\nu(t,v)  |^2_{L(H_0,H)} \leq \tilde{L}_1 |u-v|^2
+ \nu \tilde{L}_2 \|u-v\|^2, & \quad  \forall t\in [0,T], \; \forall u,v\in V,
\label{tilde-s-lip}
\end{align}
where ${\mathcal H}={\mathcal H}_{\frac{1}{4}} $ is defined by
\eqref{Halpha} and  $|\cdot |_{L(H_0,H)}$ denotes the (operator) norm
in the space $L(H_0,H)$ of all bounded linear operators from $H_0$ into $H$. Note that if $\nu=0$ the previous
growth and Lipschitz on $\tilde{\sigma}_0(t,.)$ can be stated  for $u,v\in H$.
}
\begin{remark}\label{re:s-tilde-s}
{\rm Unlike ({\bf C1}) the hypotheses concerning the
control intensity coefficient $\tilde{\s}_\nu $ involve a weaker topology (we deal with
the operator norm $|\cdot |_{L(H_0,H)}$ instead of the trace class norm
$|\cdot |_{L_Q}$). However we require in \eqref{tilde-s-b} a stronger
bound (in the interpolation space $\HH$). One can see that  the noise intensity $\sqrt{\nu}\,  \s_\nu$
satisfies  Condition ({\bf C2}) provided that in  Condition ({\bf C1}), we replace
point (i) by  $|\s_\nu(t,u)|^2_{L_Q} \leq K_0+ K_1 |u|^2+ K_{\mathcal H}  \|u\|_{\mathcal H}^2$.
Thus the class of intensities satisfying both
Conditions  ({\bf C1}) and   ({\bf C2}) when multiplied by $\sqrt{\nu} $  is wider than that those
coefficients which satisfy condition {\bf (C1)}   with $K_2=0$.
}
\end{remark}
\medskip
\par
Let  $M >0$,  $h\in {\mathcal A}_M$,  $\xi$ an $H$-valued random variable independent of $W$  and $\nu >0$.
Under Conditions   ({\bf C1}) and ({\bf C2}) we consider the nonlinear
SPDE
\begin{equation}  \label{uhnu}
d u_h^\nu(t)  + \big[ \nu\,  A u_h^\nu(t) + B\big(u_h^\nu(t) \big) \big]\, dt
 =  \sqrt{\nu}\, \sigma_\nu (t,u_h^\nu(t))\, dW(t) + \tilde{\s}_\nu(t, u_h^\nu(t)) h(t)\, dt,
\end{equation}
 with initial condition
$u_h^\nu(0)=\xi$.
Using \cite{CM}, Theorem 3.1, we know that for every $T>0$  and $\nu >0$ there exists $\bar{K}_2^\nu :=
\bar{K}_2(\nu,T,M)>0$ such that if $h_\nu\in {\mathcal A}_M$, $\EX|\xi|^4<+\infty$
and $0\leq K_2< \bar{K}_2^\nu $, equation
\eqref{uhnu} has a unique solution
$u^\nu_h \in {\mathcal C}([0,T],H)\cap L^2([0,T],V)$ which  satisfies:
\begin{align*}
(u^\nu_h,v)-(\xi,v)& +\int_0^t \big[ \nu \langle u_h^\nu(s) , Av\rangle +\langle B(u_h^\nu (s)), v\rangle \big]\,
ds
\\
& = \int_0^t \big( \sqrt{\nu}\, \s_\nu(s,u_h^\nu(s))\, dW(s)\, ,\, v\big) +
\int_0^t \big( \tilde{\s}_\nu(s,u^\nu_h(s))h(s)\, ,\, v\big)\, ds
\end{align*}
a.s. for all $v\in Dom(A)$ and $t\in [0,T]$.
Note that  $u_h^\nu$ is a weak solution from the analytical point of view, but a strong one
from the probabilistic point of view, that is written in terms of the given Brownian motion $W$. Furthermore,
if $K_2\in [0, \bar{K}_2^\nu[$ and $L_2\in [0,2[$, there exists a constant $C_\nu:=C(K_i, L_j, \tilde{K}_i,
\tilde{K}_{\mathcal H}, T,M,\nu)$ such that
\begin{equation} \label{boundini}
\EX\Big( \sup_{0\leq t\leq T} |u_h^\nu(t)|^4 + \int_0^T \|u_h^\nu(t)\|^2\, dt +
\int_0^T \|u_h^\nu(t)\|_{\mathcal H}^4 \, dt \Big) \leq C_\nu\, (1+\EX|\xi|^4).
\end{equation}
The following proposition proves that $ \bar{K}_2^\nu$ can be chosen independent of $\nu$
and that a proper formulation of upper estimates  of the $H$, ${\mathcal H}$ and $V$ norms of
the solution $u^\nu_h$ to \eqref{uhnu} can be proved uniformly in $h\in {\mathcal A}_M$ and
in  $\nu \in (0,\nu_0]$ 
for some constant $\nu_0>0$.
\begin{prop} \label{unifnu}  Fix $M>0$, $T>0$,  $\sigma_\nu$ and $\tilde{\s}_\nu$ satisfy Conditions
({\bf C1})--({\bf C2}) and let the initial condition $\xi$ be such that $\EX|\xi|^4<+\infty$.
Then in any shell model where $B$ is defined by \eqref{GOY} or \eqref{Sabra},
there exist constants $\nu_0>0$, $\bar{K}_2$ and  $\bar{C}(M)$ such that if $0<\nu\leq \nu_0$,
 $0\leq K_2<\bar{K}_2$,  $L_2<2$ and
  $h\in {\mathcal A}_M$,  the solution $u^\nu_h$ to \eqref{uhnu} satisfies:
\begin{align} \label{bound1}
 \EX \,\Big(\, & \sup_{0\leq t\leq T}|u^\nu_{h}(t)|^4  +
\nu \int_0^T
\|u^\nu_{h}(s)\|^2 \, ds + \nu \int_0^T \|u^\nu_{h}(s)\|_{\mathcal H}^4\,  ds \, \Big)
 \leq  \bar{C}(M)\,  \big( \EX|\xi|^4 +1\big).
\end{align}
\end{prop}
\begin{proof}
 For every $N>0$, set
$\tau_N = \inf\{t:\; |u^\nu_{h}(t)| \geq N \}\wedge T.$
 It\^o's
 formula and the antisymmetry relation in \eqref{antisym}
yield that for $t \in [0, T]$,
\begin{align} \label{Ito1}
 |u^\nu_{h}(t\wedge& \tau_N)|^2 =  |\xi|^2
+ 2 \sqrt{\nu}\, \int_0^{t\wedge \tau_N}\!\!
 \big( \sigma_\nu(s,u^\nu_{h}(s))  dW(s) , u^\nu_{h}(s)\big)
-2\, \nu \int_0^{t\wedge \tau_N}\!\! \|u^\nu_{h}(s)\|^2 ds
 \nonumber \\
& \, +2\int_0^{t\wedge \tau_N}\!\! \big(
 \tilde{\sigma}_\nu (s,u^\nu_{h}(s)) h(s), u^\nu_{h}(s)\big) \, ds
 +  \nu \int_0^{t\wedge \tau_N} \!\! |\sigma_\nu(s,u^\nu_{h}(s))|_{L_Q}^2\, ds,
\end{align}
and using again It\^o's formula we have
\begin{equation} \label{estimate1}
|u^\nu_{h}(t\wedge \tau_N)|^4   + 4\, \nu  \int_0^{t\wedge \tau_N} \!\! |u^\nu_{h}(r)|^2  \,
   \|u^\nu_{h}(r)\|^2 \, dr
\leq \;   |\xi|^4   + I(t) +  \sum_{1\leq j\leq 3} {T}_j(t),
\end{equation}
where
\begin{eqnarray*}
I(t) &= & 4\sqrt{\nu} \;  \Big| \int_0^{t\wedge \tau_N}
\big(\s_\nu(r,u^\nu_{h}(r))\; dW(r),
u^\nu_{h}(r)\; |u^\nu_{h}(r)|^{2}\big ) \Big| , \\
{T}_1(t) &= & 4 \,  \int_0^{t\wedge \tau_N}
|(\tilde{\s}_\nu (r,u^\nu_{h}(r))\, h(r)\,,\:  u^\nu_{h}(r))| \; |u^\nu_{h}(r)|^{2}
dr, \\
T_2(t) &= &  2\nu \,  \int_0^{t\wedge \tau_N}
|\s_\nu(r,u^\nu_{h}(r))|^2_{L_Q} \; |u^\nu_{h}(r) |^2   dr,  \\
{T}_3(t) &= & 4\nu   \,  \int_0^{t\wedge \tau_N}
|\s_\nu^*(s, u^\nu_{h}(r))\;  u^\nu_{h}(r)|^2_{0}\,  dr.
\end{eqnarray*}
Since $h\in \mathcal{A}_M$, the Cauchy-Schwarz and Young  inequalities and
 condition {\bf (C2)}   imply that for any $\epsilon >0$,
\begin{align} \label{majT1}
 {T}_1 (t)& \leq \; 4  \;   \int_0^{t\wedge \tau_N}\!\!
\Big( \sqrt{\tilde{K}_0} +
\sqrt{\tilde{K}_1} \, |u^\nu_{h}(r)| +  \sqrt{\nu\, \tilde{K}_{\mathcal H}} \,
k_0^{-\frac{1}{2}} \|u^\nu_{h}(r)\| \Big) \,|h(r)|_0 \,|u^\nu_{h}(r)|^3  dr  \nonumber \\
&  \leq \; 4\, \sqrt{\tilde{K}_0 \, M\, T}
+ 4 \Big(\sqrt{\tilde{K}_0 } + \sqrt{\tilde{K}_1}\Big)  \,  \int_0^{t\wedge \tau_N}
|h(r)|_0 \, |u^\nu_h(r)|^4\, ds \nonumber \\
& \qquad +
\epsilon \, \nu \int_0^t \|u^\nu_{h}(r)\|^2 \, |u^\nu_{h}(r)|^2 \, dr
+  \frac{4\, \tilde{K}_{\mathcal H}}{\epsilon\, k_0}  \, \int_0^{t\wedge \tau_N} \!\! |h(r)|_0^2 \,
|u^\nu_{h}(r)|^4 \, dr.
\end{align}
Using condition {\bf (C1)} we deduce
\begin{align} \label{majT2}
& {T}_2 (t)+T_3(t)  \leq  \; 6\, \nu  \,  \int_0^{t\wedge \tau_N} \!\!
\big[K_0+K_1\, |u^\nu_h(r)|^2 + K_2 \|u^\nu_{h}(r)\|^2\big]
 \ |u^\nu_{h}(r)|^2 \, dr \nonumber \\
& \quad \leq 6\, \nu\, K_0\, T +   6\, \nu\, (K_0+K_1)  \, \int_0^{t\wedge \tau_N} \!\! |u^\nu_h(r)|^4\, dr
+ 6 \, \nu \, K_2 \int_0^t \| u^\nu_h(r)\|^2 \,  |u^\nu_{h}(r)|^2  dr.
\end{align}
Let $K_2 \leq \frac{1}{2}$ and $0< \epsilon \leq  2-3K_2$; set
\[ \varphi(r)=4\Big( \sqrt{\tilde{K}_0} +  \sqrt{\tilde{K}_1}\Big)\, |h(r)|_0 +
\frac{4\tilde{K}_{\mathcal H}}{\epsilon k_0}\, |h(r)|_0^2 + 6\, \nu (K_0+K_1) .\] 
Then  a.s.
\begin{equation}\label{intphi}
\int_0^T \varphi(r)\, dr\leq 4\Big( \sqrt{\tilde{K}_0} +  \sqrt{\tilde{K}_1}\Big)\, \sqrt{M\, T}
+ \frac{4\tilde{K}_{\mathcal H}}{\epsilon k_0}\, M +  6\, \nu (K_0+K_1)\, T:=\Phi
\end{equation}
and the inequalities \eqref{estimate1}-\eqref{majT2}
 yield that for
\[ X(t)=\sup_{r\leq t} |u^\nu_h(r\wedge \tau_N)|^4\; ,\;
 Y(t)=\nu \int_0^t \|u^\nu_h(r\wedge \tau_N)\|^2\,
|u^\nu_h(r\wedge \tau_N)|^2\, ds ,\]
\begin{equation}\label{estimate2}
X(t)+(4-6K_2-\epsilon) Y(t)\leq |\xi|^4 + \Big(4 \sqrt{\tilde{K}_0 M T}+  6\nu K_0T   \Big) + I(t)
+ \int_0^t \varphi(s)\, X(s)\, ds.
\end{equation}
The Burkholder-Davis-Gundy inequality, ({\bf C1}),    Cauchy-Schwarz and Young's  inequalities 
yield   that for $t\in[0, T]$ and $\delta,\kappa >0$,
\begin{align} \label{estimate3}
& \EX  I(t)  \leq 12 \,\sqrt{\nu} \,  \EX \Big( \Big\{ \int_0^{t\wedge \tau_N}
\big[ K_0+K_1\, |u^\nu_h(s)|^2 + K_2\, \|u^\nu_h(s)\|^2\big]\, |u^\nu_{h}(r)|^6  ds \Big\}^\frac12 \Big)
     \nonumber \\
 &\; \leq  12\, \sqrt{\nu} \, \EX \Big( \sup_{0\leq s\leq t} |u^\nu_h(s\wedge\tau_N)|^2 \,
 \Big\{ \int_0^{t\wedge \tau_N}
\big[ K_0+K_1\, |u^\nu_h(s)|^2 + K_2\, \|u^\nu_h(s)\|^2\big]\, |u^\nu_{h}(s)|^2  ds \Big\}^\frac12 \Big)
  \nonumber \\
&\;  \leq  \delta \, \EX(Y(t)) +   \Big( \frac{36  K_2 }{\delta} + \kappa\, \nu\Big)   \, \EX(X(t))
 + \frac{36}{\kappa}\, \Big[ K_0\, T  + (K_0+K_1) \int_0^t \EX (X(s))\, ds\Big].
\end{align}
Thus we can apply Lemma 3.2 in \cite{CM} (see also Lemma 3.2  in \cite{DM}),
 and we deduce that for $0<\nu\leq \nu_0$, $K_2\leq \frac{1}{2}$, $\epsilon= \alpha = \frac{1}{2}$,
$  
\beta=\frac{36 K_2}{\delta} + \kappa\, \nu_0 \leq 2^{-1}\, e^{-\Phi}$, $ \delta
\leq \alpha 2^{-1}\, e^{-\Phi}$  
and $\gamma =\frac{36}{\kappa}(K_0+K_1)$,   
\begin{equation} \label{Gronwallgene}
\EX \Big( X(T)+\alpha  Y(T)\Big) \leq 2 \exp\big( \Phi +2T\gamma e^\Phi\big)
\Big[ 4\sqrt{\tilde{K}_0\, M\, T} +6 \nu_0 K_0 T + \frac{36}{\kappa} K_0T + \EX(|\xi|^4)\Big].
\end{equation}
Using the last inequality from \eqref{compcalH}, we deduce that for $K_2$ small enough, $\bar{C}(M)$
independent of $N$ and $\nu\in ]0,\nu_0]$,
\[ \EX\Big( \sup_{0\leq t\leq T} |u^\nu_h(t\wedge\tau_N)|^4 +\nu\, \int_0^{\tau_N}\|u^\nu_h(t)\|_{\mathcal H}^4
\, dt \Big) \leq \bar{C}(M) (1+\EX(|\xi|^4)).
\]
As $N\to +\infty$, the monotone convergence theorem yields that for $\bar{K}_2$ small enough and $\nu\in ]0,\nu_0]$
\[ \EX\Big( \sup_{0\leq t\leq T} |u^\nu_h(t)|^4 +\nu\, \int_0^T \|u^\nu_h(t)\|_{\mathcal H}^4
\, dt \Big) \leq \bar{C}(M) (1+\EX(|\xi|^4)).
\]
\par
This inequality and \eqref{Gronwallgene} with $t$ instead of $t\wedge\tau_N$ conclude the proof of \eqref{bound1}
by a similar simpler computation based on conditions {\bf (C1)} and {\bf (C2)}.
\end{proof}

\section{Well posedeness, more a priori bounds and inviscid equation}\label{s3}
The aim of this section is twofold. On one hand,
we deal with the inviscid case $\nu=0$ for which the PDE
\begin{equation}\label{u0h}
du_h^0(t) + B(u_h^0(t))\, dt = \tilde{\sigma}_0(t,u^0_h(t))\, h(t)\, dt\; , \quad u_h^0(0)=\xi
\end{equation}
can be solved for every $\omega$.
 In order to prove that \eqref{u0h} has a unique solution in
${\mathcal C}([0,T],V)$ a.s., we will need stronger assumptions on the constants  $\mu, a,b$
defining $B$, the initial condition $\xi$ and $\tilde{\s}_0$.
The initial condition $\xi$  has to belong to $V$  and the coefficients $a,b,\mu$ have to be chosen such that
$(B(u,u), Au)=0$ for $u\in V$ (see \eqref{B(u)AuG}).  On the other hand, under these assumptions 
and under stronger assumptions on $\s_\nu$ and $\tilde{\s}_\nu$,
 similar to that imposed on $\tilde{\s}_0$,
 we will prove further properties of $u_h^\nu$
for  a strictly positive viscosity coefficient $\nu$.

Thus, suppose furthermore that for $\nu >0$ (resp. $\nu =0$), the map
\[ \tilde{\s}_\nu : [0,T]\times Dom(A) \to L(H_0,V) \;
 (\mbox{\rm resp. } \tilde{\s}_0 : [0,T]\times V \to L(H_0,V))\]
satisfies the following:
\smallskip

\noindent {\bf Condition (C3):} {\it There exist non negative constants $\tilde{K}_i$ and $\tilde{L}_j$,
$i=0,1,2$, $j=1,2$ such that for $s\in [0,T]$ and for any  $u,v\in Dom(A)$ if $\nu>0$ (resp.
for any    $u,v\in V$ if $\nu=0$),
\begin{equation} \label{growthbis}
|A^{\frac{1}{2}} \tilde{\s}_\nu(s,u)|_{L(H_0,H)}^2 \leq \tilde{K}_0 + \tilde{K}_1 \, \|u\|^2 +
\nu\, \tilde{K}_2\, |Au|^2,
\end{equation}
and
\begin{equation} \label{Lipbis}
|A^{\frac{1}{2}} \tilde{\s}_\nu(s,u) - A^{\frac{1}{2}} \tilde{\s}_\nu(s,v)  |_{L(H_0,H)}^2 \leq
\tilde{L}_1  \, \|u-v\|^2 +
\nu\, \tilde{L}_2\, |Au-Av|^2.
\end{equation}
}
\begin{theorem}\label{exisuniq0}
Suppose that $\tilde{\s}_0$ satisfies the conditions {\bf (C2)} and {\bf (C3)}
and that  
the coefficients $a,b,\mu$ defining $B$ satisfy $a(1+\mu^2)+b\mu^2=0$.
Let $\xi\in V$ be deterministic.
For any $M>0$ there exists $C(M)$ such that  equation \eqref{u0h} has a unique
solution in ${\mathcal C}([0,T],V)$ for any $h\in {\mathcal A}_M$, and  a.s. one has:
\begin{equation}\label{boundu0}
\sup_{h\in {\mathcal A}_M}\, \sup_{0\leq t\leq T} \|u^0_h(t)\| \leq C(M)(1+\|\xi\|) .
\end{equation}
\end{theorem}
Since equation \eqref{u0h} can be considered for any fixed $\omega$, it suffices to check that
the deterministic equation  \eqref{u0h}
has a unique solution in ${\mathcal C}([0,T],V)$ for any $h\in S_M$ and that \eqref{boundu0} holds.
  For any $m\geq 1$, let $ H_m = span(\varphi_1, \cdots, \varphi_m) \subset
Dom(A)$,
\begin{equation} \label{HPm}
  P_m: H \to H_m \quad \mbox{\rm denote the orthogonal
projection from  $H$ onto $H_m$},
  \end{equation}
and finally  let
$\tilde{\s}_{0,m}=P_m \tilde{\s}_0$. Clearly  $P_m$ is a
contraction of $H$ and $|\tilde{\s}_{0,m}(t,u)|_{L(H_0,H)}^2 \leq
| \tilde{\s}_0(t,u)|_{L(H_0,H)}^2$. Set $u^0_{m,h}(0)=P_m\, \xi$ and
consider the ODE on  the $m$-dimensional space $H_m$
defined  by
\begin{equation}\label{Galerkin0}
d\big( u^0_{m,h}(t)\, ,\, v\big) = \big[-\big( B(u^0_{m,h}(t) )\,
,\, v\big) +\big( \tilde{\s}_0(t,u^0_{m,h}(t))\, h(t)\, ,v\big)
\big]\, dt
\end{equation}
for every $v\in H_m$.

Note that  using    \eqref{boundB1}  we
deduce that
  the map $  u \in H_m \mapsto \langle B(u)\, ,\, v\rangle  $ is
locally Lipschitz.
 Furthermore, since there exists some constant $C(m)$
 such that $\|u\| \vee \|u\|_{\mathcal H}\leq C(m) |u|$ for $u\in H_m$,
 Condition ({\bf C2})  implies  that the
map  $u\in H_m  \mapsto \big( (\tilde{\s}_{0,m}(t, u) h(t)\, ,\,
\varphi_k) : 1\leq k\leq m \big) $, is   globally Lipschitz from
$H_m$ to $\RR^m$ uniformly in $t$.
 Hence by a well-known result about existence and
uniqueness of solutions to ODEs, there exists a maximal solution
$u^0_{m,h}=\sum_{k=1}^m (u^0_{m,h}\, ,\, \varphi_k\big)\,
\varphi_k$ to \eqref{Galerkin0}, i.e., a (random) time
$\tau^0_{m,h}\leq T$ such that \eqref{Galerkin0} holds for $t<
\tau^0_{m,h}$ and as $t \uparrow \tau^0_{m,h}<T$, $|u^0_{m,h}(t)|
\to \infty$. The following lemma provides the (global) existence
and uniqueness of approximate solutions as well as their uniform a
priori estimates. This is the main preliminary step in the proof
of Theorem~\ref{exisuniq0}.
\begin{lemma}\label{boundGalerkin}
Suppose that the assumptions of Theorem \ref{exisuniq0} are
satisfied and fix $M>0$. Then for every $h\in {\mathcal A}_M$
 equation \eqref{Galerkin0}
has a unique solution in ${\mathcal C}([0,T],H_m)$. There exists
some constant $C(M)$ such that
 for every  $h\in {\mathcal A}_M$,
\begin{equation}\label{boundu0nh}
\sup_m\,  \sup_{0\leq t\leq T} \|u^0_{m,h}(t)\|^2 \leq C(M)\,
(1+\|\xi\|^2)  \; a.s.
\end{equation}
\end{lemma}
\begin{proof}
The proof is included for the sake of completeness; the arguments are similar to that
in the classical viscous framework. Let $h\in {\mathcal A}_M$ and let  $u^0_{m,h}(t)$ be the
approximate maximal solution to \eqref{Galerkin0} described above.
 For every $N>0$, set
$ 
\tau_N = \inf\{t:\; \|u^0_{m,h}(t)\| \geq N \}\wedge T.
$ 
Let $\Pi_m : H_0\to H_0$ denote the projection operator defined by
$\Pi_m u =\sum_{k=1}^m \big( u \, ,\, e_k\big) \, e_k$, where
 $\{e_k, k\geq 1\}$ is the orthonormal basis of $H$  made by
 eigen-elements of the covariance operator
 $Q$ and used in  \eqref{W-n}.
\par
Since $\varphi_k\in Dom(A)$ and $V$ is a Hilbert space,   $P_m$ contracts the $V$ norm and
commutes with $A$. Thus, using {\bf (C3)} and  \eqref{B(u)AuG},  we deduce
\begin{align*}
 \|u^0_{m,h}(t\wedge \tau_N)\|^2   
\leq &\; \|\xi\|^2 - 2\int_0^{t\wedge \tau_N}\!\! \big( B(u^0_{m,h}(s))\, ,\, A u^0_{m,h}(s)\big)\,  ds \\ 
& + 2  \int_0^{t\wedge \tau_N}
 \big|  A^{\frac{1}{2}} P_m \tilde{\sigma}_{0,m} (s,u^0_{m,h}(s)) h(s)\big|\, \|u^0_{m,h}(s)\|\, ds \\
\leq & \; |\xi\|^2 + 2\sqrt{\tilde{K}_0 M T} + 2\Big(
\sqrt{\tilde{K}_0}  + \sqrt{\tilde{K}_1}\Big) \int_0^{t\wedge
\tau_N} |h(s)|_0\, \|u^0_{m,h}(s)\|^2\, ds.
\end{align*}
Since the map $\|u^0_{m,h}(.)\|$ is bounded on $[0,\tau_N]$,
Gronwall's lemma implies that for every $N>0$,
\begin{equation}\label{Galerkin4}
\sup_m \sup_{t\leq \tau_N} \|u^0_{m,h}(t)\|^2  \leq \Big(
\|\xi\|^2 + 2\sqrt{\tilde{K}_0 M T}\Big) \, \exp\Big( 2\sqrt{M T}\,
\Big[ \sqrt{\tilde{K}_0 }+ \sqrt{\tilde{K}_1}\Big]\Big).
\end{equation}
 Let  $\tau:=\lim_N \tau_N$ ;  as $N\to \infty$ in \eqref{Galerkin4} we deduce
 
\begin{equation}\label{GalerkinT}
\sup_m  \sup_{t\leq \tau} \|u^0_{m,h}(t)\|^2  \leq \Big( \|\xi\|^2 + 2\sqrt{\tilde{K}_0 M T}\Big)
\, \exp\Big( 2 \sqrt{M T}\, \Big[ \sqrt{\tilde{K}_0 }+ \sqrt{\tilde{K}_1}\Big]\Big).
\end{equation}
On the other hand,  $\sup_{t\leq \tau} \|u^0_{m,h}(t)\|^2 = +\infty$ if $\tau<T$,   which 
contradicts the estimate (\ref{GalerkinT}) . Hence $\tau = T$ a.s.   and we get
\eqref{boundu0nh} which completes the proof of the Lemma.
\end{proof}

 We now prove the main result of this section.
\par
\noindent \textbf{Proof of Theorem \ref{exisuniq0}: }\\
\noindent \textbf{Step 1: }  Using the
estimate \eqref{boundu0nh} and the growth condition
\eqref{tilde-s-b} we conclude that each component of the sequence
$\big( (u^0_{m,h})_n, n\geq 1\big)$ satisfies the following
estimate
\begin{equation*}
\sup_m\,  \sup_{0\leq t\leq T} |(u^0_{m,h})_{n}(t)|^2 +
\big|\big(\tilde{\sigma}_0(t,u^0_{m,h}(t)) h(t)\big)_n\big|\leq
C\; a.s. \, , \forall  n=1, 2, \cdots
\end{equation*}
for some constant $C>0$ depending only on $M, \|\xi\|, T$.
Moreover, writing the equation \eqref{u0h}  for the GOY shell model
 in the componentwise
form  using  \eqref{GOY} (the proof for the Sabra shell model using
\eqref{Sabra}, which is similar, is omitted), we obtain for $n=1, 2, \cdots$
\begin{align}\label{GOY-u0mnh}
(u^0_{m,h})_{n}(t)=&(P_{m}\xi)_n+i\int_{0}^{t}( a k_{n+1}
(u^0_{m,h})_{n+1}^* (s)(u^0_{m,h})_{n+2}^*(s) +b k_{n}
(u^0_{m,h})_{n-1}^*(s) (u^0_{m,h})_{n+1}^*(s) \nonumber\\
&-a k_{n-1} (u^0_{m,h})_{n-1}^*(s) (u^0_{m,h})_{n-2}^*(s) -b
k_{n-1} (u^0_{m,h})_{n-2}^*(s) (u^0_{m,h})_{n-1}^*(s))ds\nonumber\\
&+\int_{0}^{t}\big( \tilde{\s}_0(s,u^0_{m,h}(s))\, h(s)\big)_n ds\, .
\end{align}
Hence, we deduce  that for every $n\geq 1$ there exists a constant
$C_{n}$, independent of $m$, such that
$$\|(u^0_{m,h})_{n}\|_{C^{1}\left([0,T];\mathbb{C}\right)}\leq
C_{n}.$$
Applying the Ascoli-Arzel\`a theorem, we conclude that for every
$n$ there exists a subsequence $(m^{n}_{k})_{k\geq 1}$ such that
$(u^0_{m^{n}_{k},h})_{n}$ converges uniformly to some
$(u^0_{h})_{n}$ as $k\longrightarrow\infty$. By a diagonal
procedure, we may choose a sequence $(m^{n}_{k})_{k\geq 1}$
independent of $n$ such that $(u^0_{m,h})_{n}$ converges uniformly
to some $(u^0_{h})_{n}\in {\mathcal C}\left([0,T];\mathbb{C}\right)$  for
every $n\geq 1$; set
$$u^0_{h}(t)=((u^0_{h})_{1}, (u^0_{h})_{2},\dots).$$
From the estimate \eqref{boundu0nh}, we have the weak star
convergence in $L^{\infty}(0,T; V)$ of some further subsequence of  $\big(
u^0_{m^{n}_{k},h}\,: \, {k\geq 1})$.  The weak limit belongs to
$L^{\infty}(0,T; V)$ and has clearly $(u^0_{h})_{n}$ as components
that belong to ${\mathcal C}\left([0,T];\mathbb{C}\right)$ for every integer
$n\geq 1$. Using the uniform convergence of each component, it is
easy to show, passing to the limit in the expression
\eqref{GOY-u0mnh}, that $u^0_{h}(t)$ satisfies the weak form of
the GOY shell model equation \eqref{u0h}. Finally, since
\[ u^0_h(t)=\xi + \int_0^t \big[ -B(u^0_h(s)) + \tilde{\sigma}_0(s,u^0_h(s)) h(s)\big]\, ds, \]
is such that $\sup_{0\leq s\leq T} \|u^0_h(s)\|<\infty$  a.s.
and since for every $s\in [0,T]$, by \eqref{boundB1} and \eqref{growthbis}
we have a.s.
\[ \big[ \|B(u^0_h(s))\|+ \|\tilde{\sigma}_0(s,u^0_h(s))
h(s)\|\big] \leq C \Big( 1+\sup_{0\leq s\leq T} \|u^0_h(s)\|^2 \big)
\big(1+|h(s)|_0\Big)\in L^2([0,T]),\]
we deduce that $u^0_h\in {\mathcal C}([0,T],V)$ a.s.
\smallskip

\noindent \textbf{Step 2: } To complete the proof of Theorem
\ref{exisuniq0}, we  show
 that
the solution  $u_h^0$ to \eqref{u0h} is unique in ${\mathcal
C}([0, T], V) $.
 Let $v \in \mathcal{C} ([0,T],V)$ be another solution to \eqref{u0h} and set
 \[
\tau_N =\inf \{t \geq 0: \|u_h^0(t) \| \geq N \} \wedge \inf \{t
\geq 0: \|v(t) \| \geq N \} \wedge T.
\]
Since $\|u_h^0(.) \|$ and $\|v(.)\|$ are
 bounded on $[0,T]$,  we have    $\tau_N \to T$  as $N\to \infty$.

Set $U=u^0_h-v$; equation \eqref{diffBV} implies
\begin{align*}
 \big| \big(  A^{\frac{1}{2}} B(u_h^0(s))- A^{\frac{1}{2}} B(v(s)),A^{\frac{1}{2}} U(s)\big)\big| &  =
\big| \big(   B(u_h^0(s))-  B(v(s)),A U(s)\big) \big| \\
& \leq  \bar{C}_1 \| U(s) \|^2\, \|v(s)\|.
\end{align*}
On the other hand, the Lipschitz property \eqref{Lipbis} from condition {\bf (C3)} for $\nu=0$ implies
\[   \big| \big[ A^{\frac{1}{2}} \tilde{\s}_0(s,u_h^0(s))- A^{\frac{1}{2}} \tilde{\s}_0(s,v(s))\big] h(s) \big|
\leq \sqrt{\tilde{L}_1}  \|u_h^0(s)-v(s)\| \, |h(s)|_0.
\]
\par
\noindent Therefore,
\begin{align*}
 \|U(t\wedge \tau_N)\|^2 \;\; & =
 \int_0^{t\wedge \tau_N}\!\!
  \Big\{  - 2 \Big(  A^{\frac{1}{2}} B(u_h^0(s))- A^{\frac{1}{2}} B(v(s)),A^{\frac{1}{2}} U(s)\Big)  \\
&\qquad \qquad  +  2 \Big( [ A^{\frac{1}{2}} \tilde{\s}_0(s,u_h^0(s))-
A^{\frac{1}{2}} \tilde{\s}_0(s,v(s))] h(s),
A^{\frac{1}{2}} U(s)  \Big) \Big\}\, ds\\
&\leq  2 \int_0^t
\big( \bar{C}_1\, N + \sqrt{L}_1|h(s)|_0\big) \, \|U(s\wedge
\tau_N)\|^2\, ds,
\end{align*}
and Gronwall's lemma implies that (for almost every $\omega$) $
\sup_{0\leq t\leq T}\|U(t\wedge \tau_N)\|^2=0$ for every $N$. As
$N\to \infty$, we  deduce that a.s. $U(t)=0$ for every $t$, which
concludes the proof. \hfill $\Box$
\bigskip

We now suppose that the diffusion coefficient $\s_\nu$ satisfies the following condition {\bf (C4)} which
strengthens {\bf (C1)} in the way {\bf (C3)} strengthens {\bf (C2)}, i.e.,

\noindent {\bf Condition (C4)} There exist constants $K_i$ and $L_i$,
$i=0,1,2$, $j=1,2$, such that for any
$\nu >0$ and $u\in Dom(A)$,
\begin{eqnarray}\label{bnd-sbis}
|A^{\frac{1}{2}}\s_\nu(s,u)|^2_{L_Q} &\leq & K_0+K_1 \|u\|^2 + K_2|Au|^2 ,\\
|A^{\frac{1}{2}}\s_\nu(s,u)-A^{\frac{1}{2}}\s_\nu(s,v) |^2_{L_Q}
&\leq & L_1 \|u-v\|^2 + L_2 |Au-Av|^2 . \label{lip-sbis}
\end{eqnarray}
Then for $\nu>0$, the existence result and apriori bounds of the solution to \eqref{uhnu} proved
in Proposition \ref{unifnu} can be improved as follows.
\begin{prop} \label{unifnu-bis}
Let $\xi\in V$,  let the coefficients $a,b,\mu$ defining $B$ be such that $a(1+\mu^2) +b\mu^2 =0$,
 $\s_\nu$ and ${\tilde \s}_{\nu}$ satisfy the conditions
{\bf (C1)}, {\bf (C2)}, {\bf (C3)} and {\bf (C4)}. Then there exist positive constants
$\bar{K}_2$ and $\nu_0$ such that for $0<K_2<\bar{K}_2$ and $0< \nu\leq \nu_0$,
for every $M>0$ there exists a constant $C(M)$ such that for any  $h\in {\mathcal A}_M$,
the solution $u^\nu_h$ to \eqref{uhnu}
belongs to ${\mathcal C}([0,T],V)$ almost surely and
\begin{equation}\label{bound-V}
\sup_{h\in {\mathcal A}_M}\, \sup_{0 < \nu\leq \nu_0} \EX \Big( \sup_{t\in [0,T]} \|u_h^\nu(t)\|^2 +
\nu \int_0^T |Au^\nu_h(t)|^2 \, dt \Big) \leq C(M).
\end{equation}
\end{prop}
\begin{proof}
Fix $m\geq 1$, let $P_m$ be defined by \eqref{HPm} and let  $u_{m,h}^\nu(t)$ be the approximate maximal
solution to the (finite dimensional) evolution equation:
$u^\nu_{m,h}(0) = P_m \xi$ and
\begin{eqnarray}\label{Galerkin-nu}
d u^\nu_{m,h}(t) &=& \big[-\nu P_m A u^\nu_{m,h}(t) - P_m
B(u^\nu_{m,h}(t)) + P_m \tilde{\s}_\nu(t,u^\nu_{m,h}(t))\, h(t)
\big] dt
\nonumber \\
&& + P_m \sqrt{\nu} \, \s_\nu (t,u^\nu_{m,h})(t) dW_m(t),
\end{eqnarray}
where $W_m$ is defined by \eqref{W-n}.  Proposition 3.3 in \cite{CM} proves that \eqref{Galerkin-nu} has a
unique  solution $u^\nu_{m,h} \in {\mathcal C}([0,T],P_m(H))$.
 For every $N>0$, set
\begin{equation*}
\tau_N = \inf\{t:\; \|u_{m,h}^\nu(t)\| \geq N \}\wedge T.
\end{equation*}
Since $P_m(H)\subset Dom(A)$, we may apply It\^o's formula to
$\|u^\nu_{m,h}(t)\|^2$. Let $\Pi_m : H_0 \to H_0$ be defined by
$\Pi_m u = \sum_{k=1}^m \big( u,e_k\big)\, e_k$ for some
orthonormal basis $\{ e_k, k\geq 1\}$ of $H$ made by eigen-vectors
of the covariance operator $Q$; then we have:
\begin{align*}
& \|u_{m,h}^\nu(t\wedge \tau_N)\|^2 = \|P_m \xi\|^2 + 2\sqrt{\nu}
\int_0^{t\wedge \tau_N}\!\!
 \big(A^{\frac{1}{2}}  P_m\s_\nu(s, u_{m,h}^\nu(s))  dW_m(s) , A^{\frac{1}{2}} u^\nu_{m,h}(s)\big)
 \\
&  \;  + \nu  \int_0^{t\wedge \tau_N} \!\!
|P_m\s_\nu(s,u^\nu_{m,h}(s))\, \Pi_m|_{L_Q}^2\, ds -2
\int_0^{t\wedge \tau_N}\!\! \big( A^{\frac{1}{2}} B(
u^\nu_{m,h}(s)) \, ,\,
 A^{\frac{1}{2}} u^\nu_{m,h}(s)\big)\,  ds  \\
& -2 \nu\!  \int_0^{t\wedge \tau_N}\!\!\!\! \big( A^{\frac{1}{2}}P_m
A u^\nu_{m,h}(s) ,
 A^{\frac{1}{2}} u^\nu_{m,h}(s)\big)  ds
+ 2\int_0^{t\wedge \tau_N}\!\!\!\! \big( A^{\frac{1}{2}} P_m
\tilde{\s}_\nu(s,u^\nu_{m,h}(s)) h(s) ,
 A^{\frac{1}{2}}  u^\nu_{m,h}(s)\big)  ds.
\end{align*}
Since the functions $\varphi_k$ are eigen-functions of $A$, we
have $A^{\frac{1}{2}} P_m = P_m A^{\frac{1}{2}}$ and hence $\big(
A^{\frac{1}{2}} P_m Au^\nu_{m,h}(s),A^{\frac{1}{2}}
u^\nu_{m,h}(s)\big) = |Au^\nu_{m,h}(s)|^2$. Furthermore, $P_m$
contracts the $H$ and the $V$ norms, and for $u\in Dom(A)$, $\big(
B(u), Au\big) =0$ by \eqref{B(u)AuG}.
 Hence  for $0<\epsilon =\frac{1}{2} (2-K_2)<  1$, using  Cauchy-Schwarz's inequality  and
the conditions {\bf (C3)} and {\bf (C4)} on the coefficients $\s_\nu$ and $\tilde{\s}_\nu$,  we
deduce
\begin{align*}
& \|u_{m,h}^\nu(t\wedge \tau_N)\|^2
 + \epsilon \nu  \int_0^{t\wedge \tau_N}\!\! \big| Au^\nu_{m,h}(s)|^2\, ds
\leq  \| \xi\|^2
+ \nu  \int_0^{t\wedge \tau_N} \!\! \big[ K_0 + K_1 \|u^\nu_{m,h}(s)\|^2\big] \, ds \\
&\quad + 2  \sqrt{\nu}  \int_0^{t\wedge \tau_N}\!\!
 \big(A^{\frac{1}{2}}  P_m \s_\nu(s, u_{m,h}^\nu(s))  dW_m(s) , A^{\frac{1}{2}} u^\nu_{m,h}(s)\big)
 \\
&  \quad  + 2 \int_0^{t\wedge \tau_N} \!\! \Big\{  \Big[
\sqrt{\tilde{K}_0} + \Big( \sqrt{\tilde{K}_0} +
\sqrt{\tilde{K}_1}\Big) \|u^\nu_{m,h}(s)\|^2 \Big] |h(s)|_0 + \frac{
\tilde{K}_2}{\epsilon} |h(s)|_0^2 \| u^\nu_{m,h}(s)\|^2 \Big\} \,
ds.
\end{align*}
For any $t\in [0,T]$ set
\begin{eqnarray*}
 I(t)&=&\sup_{0\leq s\leq t}\Big|2\sqrt{\nu}
   \int_0^{s\wedge \tau_N}\!\!
 \big(A^{\frac{1}{2}}  P_m\s_\nu(r, u_{m,h}^\nu(r))  dW_m(r)
\, ,\,  A^{\frac{1}{2}} u^\nu_{m,h}(r)\big)\Big|,\\
X(t)&=&\sup_{0\leq s\leq t} \|u^\nu_{m,h}(s\wedge\tau_N)\|^2,
\quad Y(t)= \int_0^{t\wedge\tau_N}
|Au^\nu_{m,h}(r)|^2\, dr,\\
\varphi(t)&=& 2  \Big( \sqrt{\tilde{K}_0} +  \sqrt{\tilde{K}_1}\Big) \, |h(t)|_0 + \nu K_1 +
 \frac{\tilde{K}_2}{\epsilon} |h(t)|_0^2.
 \end{eqnarray*}
Then almost surely, $\int_0^T \varphi(t)\, dt \leq \nu K_1 T +  2\big(
  \sqrt{\tilde{K}_0} +  \sqrt{\tilde{K}_1}\big) \sqrt{MT}  + \frac{ \tilde{K}_2}{\epsilon}M
:=C$.
The Burkholder-Davis-Gundy inequality,  conditions ({\bf C1}) -- {\bf (C4)},  Cauchy-Schwarz
and Young's  inequalities 
yield   that for $t\in[0, T]$ and $\beta >0$,
\begin{align*} 
\EX I(t)   & \; \leq \; 6 \sqrt{\nu} \, \EX \Big\{ \int_0^{t\wedge
\tau_N} \big|A^{\frac{1}{2}} \s_\nu(s,u_{m,h}^\nu(r))\; \Pi_m
|^2_{L_Q}\; \|u^\nu_{m,h}(s)\|^2
 ds \Big\}^\frac12    \\  
& \leq  \; \beta \;\EX  \Big(\sup_{0\leq s\leq t\wedge\tau_N}
\|u_{m,h}(s)\|^2 \Big) +  \frac{9 \nu K_1 }{\beta}   \;\EX
 \int_0^{t\wedge \tau_N} \| u_{m,h}(s)\|^2\,   ds\\ 
&\qquad   + \frac{9 \nu K_0}{\beta}\;T   + \frac{9\nu K_2}{\beta}
\EX \int_0^{t\wedge \tau_N} |Au^\nu_{m,h}(s)|^2 ds.
\end{align*}
Set $Z=\|\xi\|^2 + \nu_0 K_0 T +  2\sqrt{\tilde {K}_0 T M} $,
$\alpha = \epsilon \nu$,  $\beta=2^{-1}e^{-C}$, $K_2 < 2^{-2} e^{-2C} (9+2^{-3}e^{-2C})^{-1}$;
 the previous inequality implies
that the bounded function $X$ satisfies a.s. the inequality
\[
X(t)+ \alpha Y(t)  \leq Z+ I(t) + \int_0^t \varphi(s)\, X(s)\, ds  .
\]
Furthermore, $I(t)$ is non decreasing, such that for $0<\nu\leq \nu_0$, $\delta =
  \frac{9\nu K_2}{\beta} \leq \alpha 2^{-1} e^{-C}$  and  $\gamma = \frac{9\nu_0}{a} K_1$, one has
\[ \EX I(t)\leq \beta \EX X(t) +\gamma \EX \int_0^t X(s)\, ds +
 \delta Y(t)+ \frac{9\nu_0}{\beta} K_0 T.
\]
Lemma  3.2 from \cite{CM} implies that for  ${K}_2$  and $\nu_0$
small enough, there exists a constant $C(M,T)$ which does not depend on $m$ and $N$, and
	such that for $0<\nu\leq \nu_0$, $m \geq 1$ and $h\in {\mathcal A}_M$:
\[ 
\sup_{N>0} \, \sup_{m\geq 1} \EX\Big[ \sup_{0\leq t\leq \tau_N}
\|u^\nu_{m,h}(t)\|^2 + \nu \int_0^{\tau_N} |Au^\nu_{m,h}(t)|^2 \,
dt\Big]  < \infty.
\] 
Then, letting $N\to \infty$ and using the monotone convergence theorem, we deduce that
\begin{equation}\label{bound-Galer-bis}
\sup_{m\geq 1}\, \sup_{h\in {\mathcal A}_M}  \EX\Big[ \sup_{0\leq
t\leq T} \|u^\nu_{m,h}(t)\|^2 + \nu \int_0^T |Au^\nu_{m,h}(t)|^2
\, dt \Big] < \infty.
\end{equation}
Then using classical arguments we prove the existence of a
subsequence of $(u^\nu_{m,h}, m\geq 1)$ which converges weakly in
$L^2([0,T]\times \Omega , V) \cap L^4([0,T]\times \Omega ,
{\mathcal H})$ and converges weak-star in $L^4(\Omega,
L^\infty([0,T],H))$ to the solution $u^\nu_h$ to equation
\eqref{uhnu} (see e.g. \cite{CM}, proof of Theorem 3.1). In order
to complete the proof, it suffices to extract a further
subsequence of $(u^\nu_{m,h}, m\geq 1)$ which is weak-star
convergent to the same limit $u^\nu_h$ in $L^2(\Omega,
L^\infty([0,T],V))$ and converges weakly in $L^2(\Omega\times
[0,T], Dom(A))$; this is a straightforward consequence of
\eqref{bound-Galer-bis}. Then  as $m\to  \infty$ in
\eqref{bound-Galer-bis}, we conclude the proof of \eqref{bound-V}.
\end{proof}

\section{Large deviations}  \label{s4}
We will prove a  large deviation principle using  a    weak convergence approach
  \cite{BD00, BD07}, based on variational representations of
  infinite dimensional    Wiener processes. Let
 $\sigma : [0,T]\times V\to L_Q$ and  for every $\nu>0$ let
$\bar{\s}_\nu : [0,T]\times Dom(A) \to L_Q$ satisfy the
following condition:

\noindent {\bf Condition (C5)}: \\
{\it (i)  There exist a positive constant $\gamma$ and
 non negative  constants $\bar{C}$, $\bar{K}_0$, $\bar{K}_1$  and $\bar{L}_1$ such that
for all $u,v \in V$  and   $s, t\in [0,T]$:
\begin{align*}
| \s(t,u)|_{L_Q}^2
\leq \bar{K}_0 + \bar{K}_1 \, |u|^2 , &\quad
 \big| A^{\frac{1}{2}} \s(t,u)\big|^2_{L_Q} \leq \bar{K}_0  +\bar{K}_1\, \|u\|^2,
\\ 
|\s(t,u)- \s(t,v)|^2_{L_Q} \leq \bar{L}_1\, |u-v |^2 ,   &\quad
\big|  A^{\frac{1}{2}}\s(t,u) -  A^{\frac{1}{2}}\s(t,v)\big|^2_{L_Q}
\leq \bar{L}_1\, \|u-v\|^2 , \\ 
\big| \s(t,u) -   \s(s,u)\big|_{L_Q} & \leq C\, (1+\|u\|)\, |t-s|^\gamma.
\end{align*}
(ii) There exist     a positive constant $\gamma$ and
non negative constants $\bar{C}$, $\bar{K}_0$, $\bar{K}_{\mathcal H}$,
$\bar{K}_2$ and $\bar{L}_2$ such that for
 $\nu>0$,  $s,t\in [0,T]$ and $u,v\in Dom(A)$,
\begin{align*}
| \bar{\s}_\nu(t,u)|_{L_Q}^2
\leq \big( \bar{K}_0 + \bar{K}_{\mathcal H}  \,
 \|u\|_{\mathcal H}^2 \big) ,
 &\quad  \big| A^{\frac{1}{2}} \bar{\s}_\nu(t,u)\big|^2_{L_Q} \leq \big(\bar{K}_0 +  \bar{K}_2\, |Au|^2\big) ,
\\ 
|\bar{\s}_\nu(t,u)- \bar{\s}_\nu(t,v)|^2_{L_Q} \leq {\bar L}_2\, \|u-v \|^2 ,   &\,
\big|  A^{\frac{1}{2}}\bar{\s}_\nu(t,u) -  A^{\frac{1}{2}}\bar{\s}_\nu(t,v)\big|_{L_Q}^2
\leq \bar{L}_2\, |Au-Av|^2 , \\
\big| \bar{\s}_\nu(t,u) -   \bar{\s}_\nu(t,u)\big|_{L_Q} & \leq \bar{C}\,
(1+\|u\|)\, |t-s|^\gamma.
\end{align*}
}

Set
\begin{equation}\label{defs}
 \sigma_\nu  = \tilde{\s}_\nu =
 \s+\sqrt{\nu}  \bar{\s}_\nu\; 
 \;   \mbox{\rm for }\quad
\nu > 0, \quad  \mbox{\rm  and } \quad \tilde{\s}_0=\s.
\end{equation}
Then for $0\leq \nu \leq \nu_1$,
 the coefficients $\s_\nu$ and  $\tilde{\s}_\nu$ satisfy the conditions {\bf (C1)}-{\bf (C4)}
 with
\begin{align} & K_0=\tilde{K}_0=4\bar{K}_0,\;    K_1=\tilde{K}_1=2\bar{K}_1,\;   L_1=\tilde{L}_1=2 \bar{L}_1,
\;  \tilde{K}_2=2\bar{K}_2, \;  \tilde{K}_{\mathcal H}=2\bar{K}_{\mathcal H}, \nonumber \\
& K_2=2 \big[ \bar{K}_2 \vee \big( \bar{K}_{\mathcal H} k_0^{4\alpha-2}\big) \big] \nu_1 ,
\;      L_2=  2  \bar{L}_2 {\nu_1}\quad \mbox{\rm and } \tilde{L}_2= 2\bar{L}_2 . \label{constants}
 \end{align}
Proposition \ref{unifnu-bis} and Theorem \ref{exisuniq0} prove that for some $\nu_0\in ]0,\nu_1]$,
$\bar{K}_2$ and $\bar{L}_2$ small enough,   $0<\nu\leq \nu_0$
 (resp. $\nu=0$), $\xi\in V$  and $h_\nu\in {\mathcal A}_M$,  the
following equation has  a unique solution $ u_{h_\nu}^\nu$ (resp. $u^0_h$) in ${\mathcal C}(0,T],V)$:
$u^\nu_{h_\nu}(0)=u^0_h(0)=\xi$, and
\begin{align}
du^\nu_{h_\nu}(t) + \big[ \nu Au^\nu_{h_\nu} (t) + B(u^\nu_{h_\nu}(t))\big] dt
 & = \sqrt{\nu}\,  \s_\nu (t,u^\nu_{h_\nu}(t))\, dW(t) +
\tilde{\s}_\nu (t,u^\nu_{h_\nu}(t)) h_\nu(t) dt,
\label{equhnu} \\
du^0_h(t) + B(u^0_h(t))\, dt & = \sigma(t, u^0_h(t))\, h(t)\, dt.
\label{control}
\end{align}
Recall that for any $\alpha \geq 0$,  ${\mathcal H}_\alpha$ has been defined in \eqref{Halpha}
and is endowed with the norm $\|\cdot\|_\alpha$ defined in \eqref{Halpha}.
When $0\leq \alpha\leq \frac{1}{4}$, as $\nu \to 0$ we will establish a Large Deviation Principle (LDP)
in the set ${\mathcal C}([0,T],V)$ for the uniform convergence in time when $V$ is endowed  with
the norm $\|\,\cdot\,\|_\alpha$
for the family of distributions of the solutions $u^\nu$ to the evolution equation: $u^\nu(0)=\xi\in V$,
\begin{equation} \label{evolunu}
du^\nu(t) + \big[ \nu Au^\nu(t) + B(u^\nu(t))\big]\, dt  = \sqrt{\nu} \s_\nu (t,u^\nu(t))\, dW(t),
\end{equation}
whose existence and uniqueness in ${\mathcal C}([0,T],V)$
follows from Propositions  \ref{unifnu} and
\ref{unifnu-bis}. Unlike in \cite{Sundar}, \cite{DM}, \cite{MSS}  and \cite{CM},
the large deviations principle is not obtained in the
natural space, which is here  ${\mathcal C}([0,T],V)$ under the assumptions {\bf(C5)},  because
the lack of viscosity does not allow to prove that $u^0_h(t)\in Dom(A)$ for almost every $t$.

To obtain the LDP in the best possible space with the weak convergence approach, we need an extra
condition, which is part of condition {\bf (C5)} when $\alpha=0$, that is when ${\mathcal H}_\alpha=H$.
\smallskip

\noindent {\bf Condition (C6):} {\it Let $\alpha \in [0,\frac{1}{4}]$; there exists a constant $L_3$
such that for $u,v \in {\mathcal H}_\alpha$ and $t\in [0,1]$,
\begin{equation}\label{Lipalpha}
\big| A^{\alpha} \s(t,u)- A^{\alpha} \s(t,v)|_{L_Q} \leq L_3\, \|u-v\|_\alpha.
\end{equation}
}

Let $\mathcal{B}$ denote the  Borel $\s-$field of the Polish space
\begin{equation}\label{defX}
{\mathcal X}= {\mathcal C}([0,T],V)\quad \mbox{\rm endowed with the norm}\quad
\|u\|_{\mathcal X}=:\sup_{0\leq t\leq T} \|u(t)\|_\alpha,
\end{equation}
 where $\| \, \cdot\, \|_{_\alpha}$ is  defined by
\eqref{Halpha}.
 We at first recall some classical   definitions;
by convention the infimum over an empty set is  $ +\infty$.
\begin{defn}
   The random family
$(u^\nu )$ is said to satisfy a large deviation principle on
${\mathcal X}$  with the good rate function $I$ if the following conditions hold:\\
\indent \textbf{$I$ is a good rate function.} The function  $I: {\mathcal X} \to [0, \infty]$ is
such that for each $M\in [0,\infty[$ the level set $\{\phi \in {\mathcal X}: I(\phi) \leq M
\}$ is a    compact subset of ${\mathcal X}$. \\
 For $A\in \mathcal{B}$, set $I(A)=\inf_{u \in A} I(u)$.\\
\indent  \textbf{Large deviation upper bound.} For each closed subset
$F$ of ${\mathcal X} $:
\[
\lim\sup_{\nu\to 0}\; \nu \log \PX(u^\nu \in F) \leq -I(F).
\]
\indent  \textbf{Large deviation lower bound.} For each open subset $G$
of ${\mathcal X} $:
\[
\lim\inf_{\nu\to 0}\; \nu \log \PX(u^\nu \in G) \geq -I(G).
\]
\end{defn}
Let ${\mathcal C}_0=\{ \int_0^. h(s)ds \, :\, h\in L^2([0,T], H_0)\}
\subset {\mathcal C}([0, T], H_0)$.
Given $\xi\in V$ define ${\mathcal G}_\xi^0:
  {\mathcal C}([0, T], H_0)  \to  X $   by 
$ {\mathcal G}_\xi^0(g)=u_h^0 $ for $  g=\int_0^. h(s)ds \in {\mathcal C}_0$ and  $u_h^0$
is the solution  to the
(inviscid) control equation \eqref{control} with initial condition $\xi$,
and ${\mathcal G}^0_\xi(g)=0$ otherwise.
 The following theorem is the main result of this section.
\begin{theorem}\label{PGDunu}
Let  $\alpha \in [0,\frac{1}{4}]$, suppose that the constants $a,b,\mu$ defining $B$ are such that
$a(1+\mu^2) +b\mu^2=0$,  let  $\xi\in V$, and  let $\s_\nu$ and $\tilde{\s}_\nu$ be defined  for $\nu>0$
by \eqref{defs} with coefficients
$\s$ and $\bar{\s}_\nu$  satisfying  the conditions ({\bf C5}) and {\bf (C6)}
for this value of $\alpha$.
Then the solution  $(u^\nu)_{\nu>0}$ to \eqref{evolunu} with initial condition $\xi$
 satisfies a large deviation principle in ${\mathcal X}:={\mathcal C}([0,T],V)$ endowed with the
norm $\|u\|_{\mathcal X} =: \sup_{0\leq t\leq T} \|u(t)\|_\alpha$,
 with the good rate function
\begin{eqnarray} \label{ratefc}
 I (u)= \inf_{\{h \in L^2(0, T; H_0): \; u ={\mathcal G}_\xi^0(\int_0^. h(s)ds) \}}
 \Big\{\frac12 \int_0^T |h(s)|_0^2\,  ds \Big\}.
\end{eqnarray}
\end{theorem}
We at first prove   the following
technical lemma, which studies time increments of the solution to
the  stochastic control problem  \eqref{equhnu} which extends   both \eqref{evolunu} and \eqref{control}.

To state this  lemma,   we need the following notations.
For every integer $n$, let $\psi_n : [0,T]\to [0,T]$ denote a measurable map
such that:\;  $s\leq \psi_n(s) \leq \big(s+c2^{-n})\wedge T$ for some positive constant $c$
and  for every $s\in [0,T]$.
 Given $N>0$,  $h_\nu\in {\mathcal A}_M$,
  for  $t\in [0,T]$ and  $\nu\in [0,\nu_0]$,   let
\[ G_N^\nu(t)=\Big\{ \omega \, :\, \Big (\sup_{0\leq s\leq t}  \|u_h^\nu(s)(\omega)\|^2 \Big)\vee
\Big(  \int_0^t |A u_h^\nu(s)(\omega)|^2 ds \Big) \leq N\Big\}.\]
\begin{lemma} \label{timeincrement}
Let $a,b,\mu$ satisfy the condition $a(1+\mu^2)+b\mu^2=0$.
Let $\nu_0, M,N$  be positive constants,  $\sigma$ and $\bar{\s}_\nu$ satisfy condition {\bf (C5)},
$\s_\nu$ and  $\tilde{\sigma}_\nu$  be defined by \eqref{defs} for $\nu\in [0,\nu_0]$.
 For every $\nu \in ]0,\nu_0]$,
let  $\xi\in L^4(\Om;H)\cap L^2(\Om;V)$,  $h_\nu \in {\mathcal A}_M$     and let
 $u_{h_\nu}^\nu(t)$
denote  solution  to  \eqref{equhnu}.
For $\nu=0$, let  $\xi\in V$,   $h\in {\mathcal A}_M$,   let $u_h^0(t)$
denote be solution   to \eqref{control}.
 Then there exists a positive  constant
$C$ (depending on
$K_i, \tilde{K}_i,  L_i, \tilde{L}_i,  T, M, N, \nu_0$)
 such that:
\begin{align}  \label{timenu}
I_n(h_\nu ,\nu):&=\EX\Big[ 1_{G_N^\nu(T)} \int_0^T\!\!  \|u_{h_\nu}^\nu(s)-
u_{h_\nu}^\nu(\psi_n(s))\|^2 \, ds\Big]
\leq C\, 2^{-\frac{n}{2}} \; \mbox{ \rm for } 0<\nu\leq \nu_0, \\
I_n(h,0):&= 1_{G_N^0(T)} \int_0^T  \|u_h^0(s)-u_h^0(\psi_n(s))\|^2 \, ds
\leq C\, 2^{-n } \; \mbox{ \rm a.s. for } \nu=0.
\label{time0}
\end{align}
\end{lemma}
\begin{proof} For $\nu>0$, the proof is  close to that of Lemma 4.2 in \cite{DM}.
Let  $\nu \in ]0,\nu_0]$,  $h\in {\mathcal A}_M$,;
for any $s\in [0,T]$, It\^o's formula yields
\[
 \|u_{h_\nu}^\nu(\psi_n(s))-u_{h_\nu}^\nu(s)\|^2 =2\int_s^{\psi_n(s)} \!\!\! \big(A
\big[ u_{h_\nu}^\nu(r)-u_{h_\nu}^\nu(s)\big] , d u_{h_\nu}^\nu(r))
+\nu \int_s^{\psi_n(s)} \!\!\! |A^{\frac{1}{2}}\s (r,u_{h_\nu}^\nu (r))|^2_{L_Q}d r .
\]
 Therefore
$I_n(h_\nu,\nu)=  \sum_{1\leq i\leq 5} I_{n,i}(h_\nu,\nu)$, where
\begin{align*}
I_{n,1}(h_\nu,\nu)&=2\, \sqrt{\nu}\;  \EX\Big( 1_{G_N^\nu(T)} \int_0^T\!\! \!\! ds \int_s^{\psi_n(s)}\!  \big(
A^{\frac{1}{2}} \sigma_\nu(r, u_{h_\nu}^\nu(r)) dW(r) \, , \, A^{\frac{1}{2}}
\big[u_{h_\nu}^\nu(r)-u_{h_\nu}^\nu(s)\big]  \big)\Big) , \\
I_{n,2}(h_\nu,\nu)&={\nu}\;  \EX \Big( 1_{G_N^\nu(T)} \int_0^T  \!\!ds \int_s^{\psi_n(s)} \!\!
|A^{\frac{1}{2}} \sigma_\nu(r,u_{h_\nu}^\nu(r))|_{L_Q}^2 \, dr\Big) , \\
I_{n,3}(h_\nu,\nu)&=-2 \, \EX \Big( 1_{G_N^\nu(T)} \int_0^T  \!\! ds \int_s^{\psi_n(s)} \!\!
  \big\langle A^{\frac{1}{2}}  B( u_{h_\nu}^\nu(r))\, , \,  A^{\frac{1}{2}}\big[
 u_{h_\nu}^\nu(r)-u_{h_\nu}^\nu(s)\big] \big\rangle \, dr\Big) ,  \\
I_{n,4}(h_\nu,\nu)&=- 2 \,\nu\,   \EX \Big( 1_{G_N^\nu(T)} \int_0^T  \!\! ds \int_s^{\psi_n(s)} \!\!
 \big\langle A^{\frac{3}{2}}  \, u_{h_\nu}^\nu(r)\, , \, A^{\frac{1}{2}}\big[
u_{h_\nu}^\nu(r)-u_{h_\nu}^\nu(s)\big] \big\rangle  \, dr\Big) ,  \\
I_{n,5}(h_\nu,\nu)&=2 \,  \EX \Big( 1_{G_N^\nu(T)} \int_0^T \!\!\! ds \int_s^{\psi_n(s)} \!\!\! \big(
A^{\frac{1}{2}} \tilde{\sigma}_\nu(r,u_{h_\nu}^\nu(r)) \, h_{\nu}(r)\,  , \, A^{\frac{1}{2}}\big[
u_{h_\nu}^\nu(r)-u_{h_\nu}^\nu(s) \big] \big)\, dr\Big)  .
\end{align*}
Clearly  $G_N^\nu(T)\subset G_N^\nu(r)$ for $r\in [0,T]$.
Furthermore,   $\|u_h^\nu(r)\|^2 \vee \|u_h^\nu(s)\|^2 \le N$
 on $G_N^{\nu}(r)$ for $0\leq s\leq r\leq T$.
\par
The Burkholder-Davis-Gundy inequality and ({\bf C5}) yield for  $0< \nu \leq \nu_0$
\begin{align*}
|I_{n,1}(h_\nu,\nu)|&\leq
 6\sqrt{\nu} \int_0^T ds \; \EX \Big(
 \int_s^{\psi_n(s)} \big|A^{\frac{1}{2}}\s_\nu(r,u_{h_\nu}^\nu(r))\big|_{L_Q}^2
 1_{G_N^\nu(r)}\, \| u_{h_\nu}^\nu(r)- u_{h_\nu}^\nu(s)\|^2 \;  dr \Big)^{\frac{1}{2}} \\
&\leq  6 \sqrt{2 \nu_0 N }  \int_0^T ds \; \EX \Big(
 \int_s^{\psi_n(s)}
\big[ K_0+K_1\,  \|u_{h_\nu}^\nu(r)\|^2  + K_2\,  |A u_{h_\nu}^\nu(r)|^2 \big]\;   \, dr \Big)^{\frac{1}{2}}.
\end{align*}
 The Cauchy-Schwarz inequality  and Fubini theorem as well as \eqref{bound-V}, which holds
uniformly in  $\nu \in ]0,\nu_0]$ for small enough  fixed $\nu_0>0$, imply
\begin{align} \label{In1}
|I_{n,1}(h_\nu,\nu)|&\leq
6 \sqrt{2 \nu_0  N T } \Big[ \EX \int_0^T\!\!
\big[ K_0+K_1\, \|u_{h_\nu}^\nu(r)\|^2 +
K_2 |Au_{h_\nu}^\nu(r)|^2\big] \Big( \int_{(r-c2^{-n})\vee 0}^r ds\Big)  dr
 \Big]^{\frac{1}{2}} \nonumber \\
&\leq  C_1\, \sqrt{N}\,  2^{-\frac{n}{2}}
\end{align}
for some constant $C_1$ depending only on $K_i$, $i=0,1,2$, $L_j$, $ j=1,2$,
 $M$, $\nu_0$ and $T$.
The property ({\bf C5}) and Fubini's theorem  imply that for $0 < \nu\leq \nu_0$,
\begin{align} \label{In2}
|I_{n,2}(h_\nu,\nu)|&\leq \nu \,   \EX \Big( 1_{G_N^\nu(T)}   \int_0^T \!\! ds \int_s^{\psi_n(s)} \!\!
\big[ K_0+K_1\|u_{h_\nu}^\nu(r)\|^2 + K_2 |Au_{h_\nu}^\nu(r)|^2 \big]  dr\Big) \nonumber \\
&\leq
\nu_0  \EX \int_0^T   \big[ K_0 +K_1\, \|u^\nu_{h_\nu}(r)\|^2 + K_2 |A u_{h_\nu}^\nu(r)|^2 \big]
\, c 2^{-n}\, dr \;
\leq\;  C_1 2^{-n}
\end{align}
for some constant $C_1$ as above.
Since $\big\langle B(u), Au\big\rangle=0$ and $\| B(u)\|\leq C \|u\|^2$  for $u\in V$ by \eqref{boundB1},
 we deduce that
\begin{align} \label{In3}
|I_{n,3}&(h_\nu,\nu)| \leq 2 \EX \Big( 1_{G_N^\nu(T)}\!\! \int_0^T  \!\!\! ds
\int_s^{\psi_n(s)} \!\!\! dr \big( A^{\frac{1}{2}}  B(u^\nu_{h_\nu}(r)) \, ,\,
 A^{\frac{1}{2}} u^\nu_{h_\nu}(s)\big) \Big)
\nonumber \\
&\leq  2C \EX \Big( 1_{G_N^\nu(T)}\!\! \int_0^T  \!\!\! ds
\int_s^{\psi_n(s)} \!\!\ \|u^\nu_{h_\nu}(r)\|^2 \,
\| u^\nu_h(s)\| \, dr \Big) 
\leq 2C \, N^{\frac{3}{2}}\, T^2 2^{-n} . 
\end{align}
Using  Cauchy-Schwarz's inequality   and \eqref{bound-V} we deduce that
\begin{eqnarray} \label{In4}
I_{n,4}(h_\nu,\nu) &\leq &  2 \, \nu\,  \EX \Big( 1_{G_N^\nu(T)} \int_0^T  \!\! \! ds \int_s^{\psi_n(s)} \!\!\!
dr  \big[ -  |Au_{h_\nu}^\nu(r)|^2  +  |A u_{h_\nu}^\nu(r)|\,  |A u_{h_\nu}^\nu(s)|\big]\Big) \nonumber \\
&\leq &
 \frac{\nu}{2}\;   \EX \Big(  \int_0^T ds \; |A u^\nu_{h_\nu}(s)|^2 \,
\int_s^{\psi_n(s)} dr \Big) \leq   C_1     2^{-n}
\end{eqnarray}
for some constant $C_1$ as above.\\
 Finally,  Cauchy-Schwarz's inequality, Fubini's theorem,
 ({\bf C5})  and the definition
 of ${\mathcal A}_M$ yield
\begin{align}  \label{In5}
& |I_{n,5}(h_\nu,\nu)|\leq 2   \;  \EX \Big( 1_{G_N^\nu(T)} \int_0^T \!\! ds  \int_s^{\psi_n(s)} dr \nonumber \\
 &\quad  \qquad \big[ \tilde{K}_0 +\tilde{K}_1\|u_{h_\nu}^\nu(r)\|^2 +\nu  \tilde{K}_2 |A u_{h_\nu}^\nu(r)|^2
\big]^{\frac{1}{2}}\, |h_\nu(r)|_0 \, \|u_{h_\nu}^\nu(r)-u_{h_\nu}^\nu(s)\|\, \Big)
\nonumber \\
& \quad \leq  4 \sqrt{N}  \; \EX  \Big(  1_{G_N^\nu(T)}\,   \big( \tilde{K}_0 + \tilde{K}_1 N\big)^{\frac{1}{2}}
\int_{0}^T |h_\nu(r)|_0 \,
  \Big( \int_{(r-c2^{-n})\vee 0}^r ds \Big)\, dr \Big) \nonumber \\
& \quad \qquad + 4\sqrt{N}   \EX  \Big(  1_{G_N^\nu(T)}\,
  \sqrt{\nu_0 \tilde{K}_2} \int_0^T |Au^\nu_{h_\nu}(r)|\, |h_\nu(r)|_0
   \Big( \int_{(r-c2^{-n})\vee 0}^r ds \Big)\, dr  \Big) \nonumber \\
& \quad \leq 4\sqrt{N} \Big[  \sqrt{M T }  \big( \tilde{K}_0 + \tilde{K}_1 N\big)^{\frac{1}{2}} +
\big(\nu_0\,  \tilde{K}_2 \, NM\big)^{\frac{1}{2}}\Big]  \, c\, T 2^{-n}
\leq C(\nu_0,N,M,T)\,  2^{-n}.
\end{align}
Collecting the upper estimates from \eqref{In1}-\eqref{In5},
we conclude the proof of \eqref{timenu} for $0<\nu\leq \nu_0$.

\par
Let $h\in {\mathcal A}_M$; a similar argument for $\nu=0$ yields for almost every $\omega$
\[ 1_{G_N^0(T)} \int_0^T \int_0^T \|u^0_h(\psi_n(s)) -u^0_h(s)\|^2 \, ds \leq \sum_{j=1,2} I_{n,j}(h,0),\]
with
\begin{eqnarray*}
I_{n,1}(h,0)&=&-2 \,   1_{G_N^0(T)} \int_0^T  \!\! ds \int_s^{\psi_n(s)} \!\!
  \big\langle A^{\frac{1}{2}}  B( u_h^0(r))\, , \,  A^{\frac{1}{2}}\big[
 u_h^0(r)-u_h^0(s)\big] \big\rangle \, dr ,  \\
I_{n,2}(h,0)&=&2 \,    1_{G_N^0(T)} \int_0^T \!\! ds \int_s^{\psi_n(s)} \!\! \big(
A^{\frac{1}{2}} \tilde{\sigma}_0(r,u_h^\nu(r)) \, h(r)\,  , \, A^{\frac{1}{2}}\big[
u_h^\nu(r)-u_h^\nu(s) \big] \big)\, dr  .
\end{eqnarray*}
An argument similar to that which gives \eqref{In3}  proves
\begin{equation}\label{In01}
|I_{n,1}(h,0)|\leq C(T,N) \, 2^{-n}.
\end{equation}
 Cauchy-Schwarz's inequality  and {\bf (C5)} imply
\begin{align}\label{In02}
|I_{n,2}&(h,0)|\leq 2   \;   1_{G_N^0(T)} \int_0^T \!\! ds  \int_s^{\psi_n(s)} dr 
\big(\tilde{K}_0 +\tilde{K}_1\|u_h^0(r)\|^2
\big)^{\frac{1}{2}}\, |h(r)|_0 \, \|u_h^0(r)-u_h^0(s)\|
\nonumber \\
& \leq  4 \sqrt{N}  \;      \big( \tilde{K}_0 + \tilde{K}_1 N\big)^{\frac{1}{2}}
\int_{0}^T |h(r)|_0 \,
  \Big( \int_{(r-c2^{-n})\vee 0}^r ds \Big)\, dr
\leq C(N,M,T)\,  2^{-n}.
\end{align}
The inequalities \eqref{In01} and \eqref{In02} conclude the proof of \eqref{time0}.
\end{proof}

\par
Now we return to the setting of Theorem~\ref{PGDunu}.
Let $\nu_0\in ]0,\nu_1]$ be defined by Theorem \ref{unifnu} and Proposition \ref{unifnu-bis},
  $(h_\nu , 0 < \nu \leq \nu_0)$  be a family of random elements
taking values in the set ${\mathcal A}_M$ defined by \eqref{AM}.
Let  $u^\nu_{h_\nu}$   be
the solution of the corresponding stochastic control equation \eqref{equhnu}
 with initial condition $u^\nu_{h_\nu}(0)=\xi \in V$.
Note that $u^\nu_{h_\nu}={\mathcal G}^\nu_\xi\Big(\sqrt{\nu} \big( W_. + \frac{1}{\sqrt \nu}
  \int_0^. h_\nu(s)ds\big) \Big)$
due to the  uniqueness of the solution.
 The following proposition establishes the weak convergence of the family
$(u^\nu_{h_\nu})$ as $\nu\to 0$.
Its proof is similar to that of Proposition 4.5 in \cite{CM};
 see also Proposition 3.3 in \cite{DM}.
\begin{prop}  \label{weakconv}
Let $a,b,\mu$ be such that $a(1+\mu^2)+b\mu^2=0$.
Let $\alpha\in [0,\frac{1}{4}]$, $\s$ and $\bar{\s}_\nu$ satisfy  the conditions ({\bf C5}) and {\bf (C6)}
for this value of  $\alpha$, $\s_\nu$ and $\tilde{\s}_\nu$ be defined by \eqref{defs}.
Let $\xi $ be ${\mathcal F}_0$-measurable such that $\EX \big(|\xi|_H^4+\|\xi\|^2\big)< \infty$,
 and let $h_\nu$ converge to $h$ in distribution as random elements
taking values in ${\mathcal A}_M$, where this set is defined by \eqref{AM}
and endowed with the weak topology of the space $L_2(0,T;H_0)$.
Then  as $\nu \to 0$, the solution $u^\nu_{h_\nu}$ of \eqref{equhnu} converges in
distribution  in ${\mathcal X}$ (defined by \eqref{defX}) to the solution $u_h^0$ of \eqref{control}.
 That is, as $\nu\to 0$, the process
  ${\mathcal G}^\nu_\xi \Big(\sqrt{\nu} \big( W_. + \frac{1}{\sqrt{\nu}} \int_0^. h_\nu(s)ds\big) \Big)$
  converges in
distribution to $ {\mathcal G}^0_\xi \big(\int_0^. h(s)ds\big)$  in ${\mathcal C}([0,T],V)$ for the topology
of uniform convergence on $[0,T]$ where $V$ is endowed with  the norm $\|\,\cdot\,\|_\alpha$.
\end{prop}
\begin{proof}
Since ${\mathcal A}_M$ is a Polish space (complete separable metric space),
by the Skorokhod representation theorem, we can construct
processes $(\tilde{h}_\nu, \tilde{h}, \tilde{W})$ such that the
joint distribution of $(\tilde{h}_\nu,  \tilde{W})$ is the same as
that of $(h_\nu, W)$,  the distribution of $\tilde{h}$ coincides
with that of $h$, and $ \tilde{h}_\nu  \to \tilde{h}$,  a.s., in
the (weak) topology of $S_M$.
 Hence a.s. for every $t\in [0,T]$, $\int_0^t \tilde{h}_\nu(s) ds - \int_0^t
\tilde{h}(s)ds \to 0$ weakly in $H_0$.
To ease  notations, we will write
$(\tilde{h}_\nu, \tilde{h}, \tilde{W})=(h_\nu,h,W)$.
Let $U_\nu=u^\nu_{h_\nu}-u_h^0\in {\mathcal C}([0,T],V)$; then  $U_\nu(0)=0$ and
\begin{align}\label{difference2}
d U_\nu(t) = - \big[\nu Au^\nu_{h_\nu}(t) & +B(u^\nu_{h_\nu}(t))-B(u_h^0(t))\big]\, dt
 + \big[\s(t,u^\nu_{h_\nu}(t)) h_\nu(t) -\s(t,u_h^0(t)) h(t)\big]\,  dt
\nonumber\\
& +\sqrt{\nu} \;
\s_\nu(t,u^\nu_{h_\nu}(t))\,  dW(t) + \sqrt{\nu}\,  \bar{\s}_\nu(t, u^\nu_{h_\nu}(t))\, h_\nu(t)\, dt.
\end{align}
 On any finite time interval $[0, t]$ with $t\leq T$, It\^o's formula,
 yields for $\nu >0 $ and $\alpha\in [0,\frac{1}{2}]$:
\begin{align*}
&  \|U_\nu(t)\|_\alpha^2
 =  -2\nu \int_0^t \!\!\! \big( A^{1+\alpha} u^\nu_{h_\nu}(s) , A^\alpha U_\nu(s)\big) ds
-2\int_0^t \!\! \! \big\langle A^\alpha \big[ B(u^\nu_{h_\nu}(s))- B(u^0_{h}(s))\big] ,
A^\alpha U_\nu(s)\big\rangle ds\\
& \quad +  2\sqrt{\nu}\int_0^t \big(A^\alpha  \s_\nu(s,u^\nu_{h_\nu}(s)) dW(s)\, ,\, A^\alpha U_\nu(s) \big)
+ \nu \int_0^t|A^\alpha \s_\nu(s,u^\nu_{h_\nu}(s))|^2_{L_Q}\, ds  \\
&\quad + 2\sqrt{\nu} \int_0^t \big( A^\alpha \bar{\s}_\nu(s,u^\nu_{h_\nu}(s))\, h_\nu(s)\,,\,
A^\alpha U_\nu(s)\big)\, ds
\\
& \quad +
 2\int_0^t \big(A^\alpha \big[\s(s,u^\nu_{h_\nu}(s))h_{\nu}(s)-\s(s,u^0_h(s))\, h(s)\big]\, ,\,
A^\alpha  U_\nu(s) \big)\, ds.
\end{align*}
Furthermore, $\big( A^\alpha \bar{\s}_\nu(s,u^\nu_{h_\nu}(s)) h_\nu(s)\, ,\, A^\alpha U_\nu(s)\big)
= \big(  \bar{\s}_\nu(s,u^\nu_{h_\nu}(s)) h_\nu(s)\, ,\, A^{2 \alpha} U_\nu(s)\big)$.
The  Cauchy-Schwarz inequality, conditions {\bf (C5)} and {\bf (C6)},   \eqref{AalphaB}
and \eqref{compcalH}  yield since
 $\alpha \in [0,\frac{1}{4}]$
\begin{align}\label{total-error}
&  \|U_\nu(t)\|_\alpha^2  \leq  2\nu \int_0^t \!\! \big| A^{\frac{1}{2}+2\alpha} u^\nu_{h_\nu}(s)\big| \,
\big( \|u^\nu_{h_\nu}(s)\| + \|u^0_h(s)\|\big)\,  ds \nonumber  \\
& \quad + 2 C \int_0^t \!\! \|U_\nu(s)\|_\alpha^2 \,
\big( \|u^\nu_{h_\nu}(s)\| + \|u^0_h(s)\|\big)\,  ds
+  2\sqrt{\nu}\int_0^t \!\! \big( \s_\nu(s,u^\nu_{h_\nu}(s)) dW(s)\, ,\, A^{2\alpha} U_\nu(s) \big)
\nonumber \\
&\quad  + \nu \int_0^t\big[K_0 + K_1 \|u^\nu_{h_\nu}(s)\|^2 +  K_2 |Au^\nu_{h_\nu}(s)|^2 \big] \, ds
\nonumber \\
&\quad
+ 2\sqrt{\nu} \int_0^t \Big[  \sqrt{\tilde{K}_0} +
  k_0^{-\frac{1}{2}} \sqrt{\tilde{K}_{\mathcal H}} \|u^\nu_{h_\nu}(s)\|\Big]\,
|h_\nu(s)|_0\, k_0^{4\alpha -1}
\big( \|u^\nu_{h_\nu}(s)\|+ \|u^0_h(s)\|\big)\, ds \nonumber \\
&\quad +  2\int_0^t \big(A^\alpha \big[\s(s,u^\nu_{h_\nu}(s))-\s(s,u^0_h(s))\big] h_\nu(s) \, ,\,
A^\alpha  U_\nu(s) \big)\, ds \nonumber \\
& \quad +  2\int_0^t \big(A^\alpha  \s(s,u^0_h(s)) \, \big[ h_\nu(s)-h_0(s)\big]\, ,\, A^\alpha U_\nu(s)\big)\, ds
\nonumber \\ &
\leq 2  \int_0^t \!\! \|U_\nu(s)\|^2_\alpha \,  \big[C \|u^\nu_{h_\nu}(s)\|^2 +C \|u^0_h(s)\|^2 +
 L_3 |h_\nu(s)|_0\big] \, ds + \sum_{1\leq  j\leq 5} T_j(t,\nu),
\end{align}
where using again the fact that  $\alpha\leq \frac{1}{4}$, we have
\begin{align*}
T_1(t,\nu)&=2\nu \, \sup_{s\leq t} \big[ \|u^\nu_{h_\nu}(s)\| + \|u^0_h(s)\|\big]\, \int_0^t
|A u^\nu_{h_\nu}(s)|\, ds , \\
T_2(t,\nu)&= 2\sqrt{\nu}\int_0^t \big(  \s_\nu(s,u^\nu_{h_\nu}(s))  dW(s)\, ,\, A^{2\alpha} U_\nu(s)  \big), \\
T_3(t,\nu)&=  \nu   \int_0^t  \big[K_0 + K_1 \|u^\nu_{h_\nu}(s)\|^2 +  K_2 |Au^\nu_{h_\nu}(s)|^2 \big] \, ds,\\
T_4(t,\nu)&=  2\sqrt{\nu} \, k_0^{2\alpha -1}  \int_0^t \!\!  \Big[  \sqrt{\tilde{K}_0} +
  k_0^{-\frac{1}{2}} \, \sqrt{\tilde{K}_{\mathcal H}}\, \|u^\nu_{h_\nu}(s)\|\Big]  \, |h_\nu(s)|_0
\big( \|u^\nu_{h_\nu}(s)\|+ \|u^0_h(s)\|\big) ds,\\
T_5(t,\nu)&= 2\int_0^t \Big( \s(s,u^0_h(s))\, \big( h_{\nu}(s)-h(s)\big), \, A^{2\alpha} U_\nu(s)\Big)\, ds.
\end{align*}
We want  to show that as $\nu \to 0$,  $\sup_{t\in [0,T]} \|U_\nu(s)\|_\alpha \to 0$ in
probability, which implies
 that $u^\nu_{h_\nu} \to u^0_h$  in distribution in $X$. 
 Fix $N>0$ and for $t\in [0,T]$ let
\begin{eqnarray*}  G_N(t)&=&\Big\{ \sup_{0\leq s\leq t} \|u^0_h(s)\|^2 \leq N\Big\},  \\
G_{N,\nu}(t)&=&  G_N(t)\cap \Big\{ \sup_{0\leq s\leq t} \|u^\nu_{h_\nu}(s)\|^2 \leq N\Big\} \cap
 \Big\{ \nu \int_0^t |A u_{h_\nu}(s)|^2  ds \leq N
\Big\}.
\end{eqnarray*}
The proof consists in two steps.\\
\textbf{Step 1:}
For  $\nu_0 >0$ given by Proposition \ref{unifnu-bis} and Theorem \ref{exisuniq0}, we have
\[ \sup_{0<\nu\leq \nu_0}\; \sup_{h,h_\nu \in {\mathcal A}_M}
\PX(G_{N,\nu}(T)^c )\to 0 \quad\mbox{\rm as }\;  N\to+\infty.\]
Indeed, for $\nu \in ]0,\nu_0]$, $h,h_\nu \in {\mathcal A}_M$, the Markov inequality and
the a priori estimates \eqref{boundu0} and   \eqref{bound-V}, which holds uniformly in $\nu\in ]0,\nu_0]$, imply
that for $0<\nu\leq \nu_0$,
\begin{align}\label{G-c}
  \PX ( G_{N,\nu}(T)^c )
& \leq  \frac{1}{N}
  \sup_{h, h_\nu \in {\mathcal A}_M}
 \EX
\Big( \sup_{0\leq s\leq T} \|u^0_h(s)\|^2 + \sup_{0\leq s\leq T}
\|u^\nu_{h_\nu}(s)|^2
+ \nu \int_0^T|A u^\nu_{h_\nu}(s)|^2 \, ds \Big) \nonumber \\
&\leq  {C \, \big(1+ \EX|\xi|^4 + \EX\|\xi\|^2 \big)}{N}^{-1},
\end{align}
for some constant $C$ depending on $T$ and $M$, but independent of $N$ and $\nu$.

\noindent \textbf{Step 2:} Fix $N>0$, let
 $h, h_\nu \in {\mathcal A}_M$  be  such that  $h_\nu\to h$ a.s. in the weak topology
 of $L^2(0,T;H_0)$   as $\nu \to 0$. Then  one has: 
\begin{equation} \label{cv1}
\lim_{\nu\to 0} \EX\Big[ 1_{G_{N,\nu}(T)}  \sup_{0\leq t\leq T } \|U_\nu(t)|_\alpha^2 \Big] =  0.
\end{equation}
Indeed,  \eqref{total-error}  and Gronwall's lemma imply that on $G_{N,\nu}(T)$, one has for $0<\nu\leq \nu_0$:
\begin{equation} \label{*}
 \sup_{0\leq t\leq T} \|U_\nu(t)\|_\alpha^2 \leq
  \exp\Big(4NC  +2L_3 \sqrt{MT}\Big)
\sum_{1\leq j\leq 5} \;  \sup_{0\leq t\leq T}  T_j(t,\nu) \, .
\end{equation}
 Cauchy-Schwarz's  inequality  implies that for some constant $C(N,T)$ independent on $\nu$:
\begin{align}\label{T1n}
\EX \Big( 1_{G_{N,\nu}(T)} \sup_{0\leq t\leq T}|T_1(t,\nu)|\Big) &\leq 4 \sqrt{TN}\, \sqrt{\nu}\,
 \EX\Big(  1_{G_{N,\nu}(T)}
\Big\{ \int_0^T |Au^\nu_{h_\nu}(s)|^2\, ds \Big\}^{\frac{1}{2}} \Big) \nonumber \\
& \leq C(N,T) \,  \sqrt{\nu}.
\end{align}
Since the sets $G_{N,\nu}(.)$ decrease,
the Burkholder-Davis-Gundy inequality, $\alpha\leq \frac{1}{4}$, the inequality  \eqref{compcalH}
  and  ({\bf C5}) imply that for some constant $C(N,T)$ independent of $\nu$:
\begin{align}
&\EX \Big(\! 1_{G_{N,\nu}(T)}  \sup_{0\leq t\leq T} |T_2(t,\nu)|\Big)   \leq  6\sqrt{\nu} \; \EX
\Big\{\!\! \int_0^T 1_{G_{N,\nu}(s)} \, k_0^{4(2\alpha - \frac{1}{2})}\,  \|U_\nu(s)\|^2 \;
 |\s_\nu(s, u^\nu_{h_\nu}(s))|^2_{L_{Q}}
 ds\Big\}^\frac12 \nonumber \\
&\leq 
 6   \sqrt{\nu} k_0^{2(2\alpha - \frac{1}{2})} \EX  \Big\{ \!\!\int_0^T\!\! \! 1_{G_{N,\nu}(s)} 4N\,
 (K_0+  K_1\| u^\nu_{h_\nu}(s) \|^2 + K_2   |Au^\nu_{h_\nu}(s)|^2 )  ds\Big\}^\frac{1}{2} 
 \leq  C(T,N)    \sqrt{\nu}. \label{T2n}
\end{align}
The  Cauchy-Schwarz inequality  implies   
\begin{equation}\label{T3n}
\EX \Big( 1_{G_{N,\nu}(T)} \sup_{0\leq t\leq T} |T_4(t,\nu)|\Big) \leq \sqrt{\nu} \, C(N,M,T).
\end{equation}
The definition of $G_{N,\nu}(T)$ implies  that
\begin{equation}\label{T4n}
\EX \Big(  1_{G_{N,\nu}(T)} \sup_{0\leq t\leq T} |T_3(t,\nu)|\Big) \leq  C\, T\, N\,{\nu}.
\end{equation}
The inequalities \eqref{*} - \eqref{T4n} show that the proof of \eqref{cv1} reduces to check that
\begin{equation} \label{T5n}
\lim_{\nu\to 0}\; \EX \Big( 1_{G_{N,\nu}(T)} \sup_{0\leq t\leq T} |T_5(t,\nu)|\Big)=  0\, .
\end{equation}
 In further estimates we use Lemma~\ref{timeincrement} with $\psi_n=\bar{s}_n$, where
$\bar{s}_n$ is the step function defined by $\bar{s}_n = kT2^{-n}$ for $(k-1)T2^{-n} \leq s<kT2^{-n}$.
For any $n,N \geq 1$, if we set $t_k=kT2^{-n}$ for $0\leq k\leq 2^n$,  we obviously have
\begin{equation}\label{t5}
\EX\Big( 1_{G_{N,\nu}(T)}\sup_{0\leq t\leq T} |T_5(t,\nu)| \Big)
\leq 2\;  \sum_{1\leq i\leq 4}  \tilde{T}_i(N,n, \nu)+ 2 \; \EX \big( \bar{T}_5(N,n,\nu)\big),
\end{equation}
 where
\begin{align*}
\tilde{T}_1(N,n,\nu)=& \EX \Big[ 1_{G_{N,\nu}(T)} \sup_{0\leq t\leq T}  \Big| \int_0^t  \Big(  \s(s,u^0_h(s))
  \big( h_\nu(s)-h(s)\big) \, ,\,
A^{2\alpha}\big[ U_\nu(s)- U_\nu(\bar{s}_n)\big] \Big)  ds\Big| \Big] ,\\
\tilde{T}_2 (N,n,\nu)=& \EX\Big[ 1_{G_{N,\nu}(T)} \\
& \quad
\times\sup_{0\leq t\leq T} \Big| \int_0^t
\Big( [\s(s,u^0_h(s)) - \s(\bar{s}_n, u^0_h(s))] (h_\nu(s)-h(s))\, ,\, A^{2\alpha} U_\nu(\bar{s}_n)\Big)ds \Big|\Big], \\
\tilde{T}_3 (N,n,\nu)=& \EX\Big[  1_{G_{N,\nu}(T)}
\\ &  \times
\sup_{0\leq t\leq T} \Big| \int_0^t \Big( \big[ \s(\bar{s}_n,u^0_h(s))
- \s(\bar{s}_n, u^0_h(\bar{s}_n))\big]
\big(h_\nu(s) - h(s) \big)\, ,\, A^{2\alpha} U_\nu(\bar{s}_n)\Big) ds\Big| \Big] ,\\
\tilde{T}_4(N,n,\nu)=& \EX \Big[  1_{G_{N,\nu}(T)} \sup_{1\leq k\leq 2^n}  \sup_{t_{k-1}\leq t\leq t_k}
\Big| \Big( \sigma(t_k, u^0_h(t_k))  \int_{t_{k-1}}^t \!\!\!
(h_\nu(s)-h(s)) ds \, ,\,A^{2\alpha} U_\nu(t_k)\Big) \Big| \Big],\\
\bar{T}_5(N,n, \nu)=&  1_{G_{N,\nu}(T)} \sum_{1\leq k\leq 2^n} \Big| \Big( \s(t_k, u^0_h(t_k))
\int_{t_{k-1}}^{t_k }  \big(h_\nu(s)-h(s)\big)\, ds \, ,\,
A^{2\alpha} U_\nu(t_k )\Big) \Big| .
\end{align*}
Using  the Cauchy-Schwarz and Young  inequalities,  ({\bf C5}), \eqref{compcalH},
 \eqref{timenu}  and \eqref{time0} in Lemma~\ref{timeincrement}
with $\psi_n(s)=\bar{s}_n$, we deduce that  for some constant
$\bar{C}_1:= C(T,M,N)$ independent of    $\nu \in ]0, \nu_0]$,
\begin{align} \label{eqT1}
&  \tilde{T}_1(N,n,\nu)\leq
k_0^{4\alpha-1} \EX\Big[ 1_{G_{N,\nu}(T)}  \int_0^T\!\!\!
 \big( \bar{K}_0+\bar{K}_1|u^0_h(s)|^2\big)^{\frac{1}{2}}
|h_\nu(s)-h(s)|_0\, \big\| U_\nu(s)-U_\nu(\bar{s}_n)\big\|\, ds\Big]  \nonumber \\
& \quad   \leq k_0^{4\alpha-1}
\Big(  \EX \Big[ 1_{G_{N,\nu}(T)}  \int_0^T 2 \big\{ \|u^\nu_{h_\nu}(s) - u^\nu_{h_\nu}(\bar{s}_n)\|^2 +
 \|u^0_{h}(s) - u^0_{h}(\bar{s}_n)\|^2 \big\}\, ds\Big] \Big)^{\frac{1}{2}}
\nonumber \\
&\qquad \times
\sqrt{\bar{K}_0 + k_0^{-2} \bar{K}_1 N}\;
  \Big( \EX  \int_0^T  2\big[ |h_\nu(s)|_0^2 + |h(s)|_0^2\big] \, ds \Big)^{\frac{1}{2}}
\leq \bar{C}_1 \;  2^{-\frac{n}{4}}.
\end{align}
A similar computation based on ({\bf C5}) and \eqref{time0} from  Lemma \ref{timeincrement} yields  for
some constant $\bar{C}_3:=C(T,M,N)$ and any   $\nu \in ]0, \nu_0]$
\begin{align} \label{eqT2}
 \tilde{ T}_3 (N,n,\nu) &\leq  \sqrt{2Nk_0^{-2}L_1} \Big( \EX \Big[ 1_{G_{N,\nu}(T)}  \int_0^T\!\!
 \|u^0_{h}(s) - u^0_{h}(\bar{s}_n)\|^2 \, ds\Big]  \Big)^{\frac{1}{2}} \Big(
 \EX \int_0^T \!\! |h_\nu(s)-h(s)|_0^2 \,  ds \Big)^{\frac{1}{2}}
 \nonumber \\
& \leq \bar{C}_3 \;  2^{- \frac{n}{4}}.
\end{align}
The H\"older regularity {\bf (C5)} imposed on $\s(.,u)$ and the Cauchy-Schwarz inequality imply that
\begin{equation}  \label{eqHolder}
\tilde{T}_2(N,n,\nu)\leq C  \sqrt{N}  2^{-n\gamma}\, \EX\Big(
1_{G_{N,\nu}(T)} \int_0^T\!\! \left(1+\| u^0_h(s)\|\right)  |h_\nu(s)-h(s)|_0 ds \Big)
\leq \bar{C}_2  2^{-n\gamma}
\end{equation}
for some constant $\bar{C}_2=C(T,M,N)$.
Using  Cauchy-Schwarz's inequality  and ({\bf C5}) we deduce  for $\bar{C}_4=C(T,N,M)$
and any $\nu \in ]0, \nu_0]$
\begin{align} \label{eqT3}
\tilde{T}_4 & (N,n,\nu)\leq  \EX \Big[  1_{G_{N,\nu}(T)} \sup_{1\leq k\leq 2^n}
\big(\bar{K}_0+\bar{K}_1| u^0_h(t_k)|^2
\big)^{\frac{1}{2}} \int_{t_{k-1}}^{t_k}\!\! |h_\nu(s)-h(s)|_0
\, ds \, \|U_\nu(t_k)\| \, k_0^{4\alpha -1}\Big]\nonumber \\
&\leq C(N) \;  \EX\Big( \sup_{1\leq k\leq 2^n}
 \int_{t_{k-1}}^{t_k} \big(|h_\nu(s)|_0 + |h(s)|_0\big) \, ds\Big)
\leq   \bar{C}_4\;  2^{-\frac{n}{2}}.
\end{align}
Finally, note that the weak convergence of $h_\nu$ to $h$ implies that as $\nu\to 0$,  for any $a,b\in [0,T]$,
 $a<b$, the integral $\int_a^b h_\nu(s) ds \to \int_a^b h(s) ds$ in the weak topology of $H_0$.
Therefore, since
the operator $\sigma(t_k, u^0_h(t_k))$ is compact from $H_0$ to $H$, we deduce that
for every $k$,
\[
\Big| \sigma(t_k, u^0_h(t_k)) \Big(   \int_{t_{k-1}}^{t_k} h_\nu(s) ds - \int_{t_{k-1}}^{t_k} h(s) ds
  \Big) \Big|_H \to 0~~\mbox{ as }~~\nu \to 0.
\]
Hence a.s. for fixed $n$ as $\nu \to 0$, $\bar{T}_5 (N,n,\nu) \to 0$ while
 $\bar{T}_5(N,n,\nu)
\leq C(\bar{K}_0,\bar{K}_1,N,n, M)$. The dominated convergence theorem proves that
 $\EX(\bar{T}_5(N,n,\nu)) \to 0$ as $\nu \to 0$ for any fixed $n,N$.
\par
This convergence and \eqref{t5}--\eqref{eqT3} complete the proof of \eqref{T5n}.
Indeed, they imply that for any fixed $N\geq 1$  and any integer $n\geq 1$
\begin{equation*}
\limsup_{\nu\to 0}\EX \Big[ 1_{G_{N,\nu}(T)} \sup_{0\leq t\leq T}  |T_5(t,\nu)|\Big] \leq
C_{N,T,M}\;  2^{-n( \frac{1}{4}\wedge \gamma )}.
\end{equation*}
for some constant $C(N,T,M)$ independent of $n$.
Since $n$ is arbitrary, this yields for any integer $N\geq 1$ the convergence property \eqref{T5n} holds.
By the Markov inequality, we have for any $\delta >0$
\[
\PP\Big(\sup_{0\leq t\leq T} \|U_\nu(t)\|_\alpha  > \de \Big)  \leq  \PP(G_{N,\nu}(T)^c )+ \frac{1}{\delta^2}
 \EX\Big( 1_{G_{N,\nu}(T)} \sup_{0\leq t\leq T} \|U_\nu(t) \|_\alpha^2\Big).
\]
Finally,  \eqref{G-c} and \eqref{cv1} yield that for any integer $N\geq 1$,
\[
\limsup_{\nu\to 0} \PP\Big(\sup_{0\leq t\leq T} \|U_\nu(t)\|_\alpha  > \de )\le  C(T,M,\delta) N^{-1},
\]
for some constant $C(T,M,\delta)$ which does not depend on $N$.
Letting $N\to +\infty$   concludes the proof of the proposition.
\end{proof}

The following compactness result
is the second ingredient which allows to transfer the  LDP  from $\sqrt{\nu} W$ to $u^\nu$.
 Its proof is similar to that of Proposition \ref{weakconv}
and easier; it will be sketched (see also \cite{DM}, Proposition 4.4).
\begin{prop}  \label{compact}
 Suppose that the constants $a, b,\mu$ defining $B$ satisfy the condition $a(1+\mu^2)+b\mu^2=0$,
$\s$ satisfies the conditions ({\bf C5}) and ({\bf C6}) and let  $\alpha\in [0,\frac{1}{4}]$.
Fix  $M>0$, $\xi\in V$ and let
$ K_M=\{u_h^0  :  h \in S_M \}$,
where $u_h^0$ is the unique solution in ${\mathcal C}([0,T],V)$ of the deterministic control
equation \eqref{control}.
Then $K_M$ is a compact subset of  ${\mathcal X}= {\mathcal C}([0,T], V)$
endowed with the norm $\|u\|_{\mathcal X} = \sup_{0\leq t\leq T} \|u(t)\|_\alpha$.
\end{prop}
\begin{proof}
To ease  notation, we skip the superscript 0 which refers to the inviscid case.
By Theorem \ref{exisuniq0}, $K_M\subset {\mathcal C}([0,T],V) $.
Let $\{u_n\}$ be a sequence in $K_M$, corresponding to solutions of
(\ref{control}) with controls $\{h_n\}$ in $S_M$:
\begin{eqnarray*} 
d u_n(t) + B(u_n(t))dt =\s(t,u_n(t)) h_n(t) dt, \;\;
u_n(0)=\xi.
\end{eqnarray*}
 Since $S_M$ is a
bounded closed subset in the Hilbert space $L^2(0, T; H_0)$, it
is weakly compact. So there exists a subsequence of $\{h_n\}$, still
denoted as $\{h_n\}$, which converges weakly to a limit $h
\in L^2(0, T; H_0)$. Note that in fact $h \in S_M$ as $S_M$ is
closed. We now show that the corresponding subsequence of
solutions, still denoted as  $\{u_n\}$, converges in $X$ 
 to $u$ which is the solution of the
following ``limit'' equation
\begin{eqnarray*}
d u(t) + B(u(t)) dt =\s(t, u(t)) h(t) dt, \;\;
u(0)=\xi.
\end{eqnarray*}
Note that we know from Theorem \ref{exisuniq0} that $u\in {\mathcal C}([0,T],V)$, and that one only needs
to check that the convergence of $u_n$ to $u$ holds uniformly in time for the  weaker
$\|\,\cdot\,\|_\alpha$ norm on $V$.
  To ease notation
 we will often drop the time parameters $s$, $t$, ... in the equations and integrals.
Let $U_n=u_n-u$; using \eqref{AalphaB} and {\bf (C6)}, we deduce that  for $t\in [0, T]$,
\begin{align}
& \|U_n(t)\|_\alpha^2
= - 2\int_0^t \!\! \big( A^\alpha B(u_n(s))-A^\alpha B(u(s))\, ,\,  A^\alpha U_n(s)\big) \, ds
 \nonumber \\
&\qquad    + 2\int_0^t \Big\{ \Big( A^\alpha \big[ \s(s,u_n(s))
-\s(s,u(s))\big] h_{n}(s)\,,\, A^\alpha U_n(s)\Big) \nonumber \\ &
\qquad
 + \big(A^\alpha \s(s,u(s)) \big(h_n(s)-h(s)\big)\, ,\, A^\alpha U_n(s)\big)\Big\} ds  \nonumber \\
& \leq   2 C \int_0^t\!\!  \|U_n(s)\|_\alpha^2 \big( \|u_n(s)\|+\|u(s)\|\big)  ds
 + 2 L_3 \int_0^t \|U_n(s)\|_\alpha^2
  |h_n(s)|_0\, ds
\nonumber \\
&\qquad
 + 2 \int_0^t \Big(  \s(s,u(s))\, [h_{n}(s)-h(s)]\; ,\; A^{2\alpha} U_n(s)\Big) \, ds  .
\label{error1}
\end{align}
The inequality \eqref{boundu0} implies that there exists a finite positive constant $\tilde{C}$ such that
\begin{equation}\label{est-uni}
  \sup_n  \sup_{0\leq t\leq T} \big(\|u(t)\|^2 + \|u_n(t)\|^2\big)  =  \tilde{C}.
\end{equation}
Thus Gronwall's lemma implies that
\begin{equation} \label{errorbound}
 \sup_{0\leq  t\leq T} \|U_n(t)\|_\alpha^2   \leq
\exp\Big(2 C  \tilde{C}  +2  L_3 \sqrt{ M T}  \big)\Big)\,
  \sum_{1\leq i\leq 5} I_{n,N}^{i}  ,
\end{equation}
where,  as in the proof of  Proposition~\ref{weakconv}, we have:
\begin{eqnarray*}
I_{n,N}^1&=& \int_0^T \big| \big( \s(s,u(s))\,  [h_n(s)- h(s)]\, ,\,
A^{2\alpha}  U_n(s)-A^{2\alpha} U_n(\bar{s}_N)\big)\big|\, ds,
\\
I_{n,N}^2&=& \int_0^T\Big|  \Big( \big[ \s(s,u(s)) - \s(\bar{s}_N, u(s)) \big]  [h_n(s)-h(s)]\, ,\,
A^{2\alpha} U_n(\bar{s}_N)\Big)\Big| \, ds,  \\
I_{n,N}^3&=& \int_0^T\Big|  \Big( \big[ \s(\bar{s}_N,u(s)) - \s(\bar{s}_N,u(\bar{s}_N))\big]
  [h_n(s)-h(s)]\, ,\,
A^{2\alpha}  U_n(\bar{s}_N)\Big)\Big| \, ds,  \\
I_{n,N}^4&=& \sup_{1\leq k\leq 2^N} \sup_{t_{k-1}\leq t\leq t_k} \Big|\Big( \s(t_k,u(t_k ))
 \int_{t_{k-1 }}^t (h_n(s)-h(s))
ds \; ,\; A^{2\alpha} U_n(t_k) \Big)\Big| , \\
I_{n,N}^5&=& \sum_{1\leq k\leq 2^N} \Big| \Big( \s(t_k, u(t_k )) \,  \int_{t_{k-1}}^{t_k }
[ h_n(s)-h(s)]\, ds   \; ,\; A^{2\alpha}  U_n(t_k ) \Big)\Big| .
\end{eqnarray*}
The Cauchy-Schwarz inequality, \eqref{est-uni}, ({\bf C5})  and \eqref{time0}   
 imply that
for some constants   $C_i$, $i=0,\cdots,4$,  which depend  on $k_0$,  $ \bar{K}_i$,  $\bar{L}_1$,  
 $\tilde{C}$, $M$ and $T$, but do not
depend on $n$ and $N$,
\begin{align} \label{estim1}
I_{n,N}^1 &\leq C_0 \Big( \int_0^T \!\! \big( \|u_n(s)-u_n(\bar{s}_N)\|^2 + \|u(s)-u(\bar{s}_N)\|^2\big) ds
 \Big)^{\frac{1}{2}} \Big( \int_0^T\!\! |h_n(s)-h(s)|_0^2 ds \Big)^{\frac{1}{2}}
\nonumber\\
&\leq C_1  \; 2^{- \frac{N }{2}} \, ,\\
I_{n,N}^3 &\leq C_0 \Big(\int_0^T \|u(s)-u(\bar{s}_N)\|^2 ds \Big)^{\frac{1}{2}}
  2 \sqrt{M} 
 \leq C_3\;  2^{-\frac{N }{2}}\, ,
\label{estim2}\\
I_{n,N}^4 &\leq C_0 \,  2^{-\frac{N}{2}}  \Big(1+  \sup_{0\leq t\leq T}  \|u(t)\|\Big)\,  \sup_{0\leq t\leq T}
\big( \|u(t)\|+\|u_n(t)\|\big) 2\sqrt{M}  \leq
C_4\;  2^{-\frac{N}{2}}\,  .
\label{estim3}
\end{align}
Furthermore, the H\"older regularity of $\s(.,u)$ from condition {\bf (C5)} implies that
\begin{align} \label{eqHolderdet}
I_{n,N}^2 \leq& \bar{C} 2^{-N\gamma}\, \sup_{0\leq t\leq T}\big(\|u(t)\|+\|u_n(t)\|\big)\nonumber \\
&\quad \times  \int_0^T (1+\|u(s)\|)  (|h(s)|_0+
|h_n(s)|_0)\, ds  \leq C_2\, 2^{-N \gamma}.
\end{align}
For fixed $N$ and $k=1, \cdots, 2^N$,
 as $n\to \infty$, the weak convergence of $h_n$ to $h$ implies
that of $\int_{t_{k-1}}^{t_k} (h_n(s)-h(s))ds$ to 0 weakly in $H_0$.
 Since $\s(t_k,u(t_k))$ is a compact operator,
we deduce that for fixed $k$ the sequence $\s(t_k, u(t_k)) \int_{t_{k-1}}^{t_k} (h_n(s)-h(s))ds$
converges to 0 strongly in $H$ as $n\to \infty$.
Since $\sup_{n,k}   \|U_n(t_k)\|\leq 2 \sqrt{\tilde{C}}$, we have
 $\lim_n I_{n,N}^5=0$. Thus
 \eqref{errorbound}--
\eqref{eqHolderdet} yield  for every integer $N\geq 1$
\[
\limsup_{n \to \infty}  \sup_{ t\leq T} \|U_n(t)\|_\alpha^2  \le C 2^{-N ( \frac{1}{2}\wedge \gamma)}.
\]
Since $N$ is arbitrary, we deduce that $\sup_{0\leq t\leq T} \|U_n(t)\|_\alpha \to 0$ as $n\to \infty$.
This shows that every sequence in $K_M$ has a convergent
subsequence. Hence $K_M$ is  a sequentially relatively compact subset of
${\mathcal X}$. 
Finally, let $\{u_n\}$ be a sequence of elements of $K_M$ which converges to $v$ in ${\mathcal X}$.
 The above argument
shows that there exists a subsequence $\{u_{n_k}, k\geq 1\}$
which converges to some element $u_h\in K_M$
for the  uniform topology on  ${\mathcal C}([0,T], V)$ endowed with the $\|\, \cdot\, \|_\alpha$ norm.
 Hence $v=u_h$,
$K_M$ is a closed subset of ${\mathcal X}$, and this completes the proof of the proposition.
\end{proof}

\noindent {\bf Proof of Theorem \ref{PGDunu}:}
 Propositions \ref{compact} and  \ref{weakconv} imply that the family
$\{u^\nu\}$ satisfies the Laplace principle,  which is equivalent
to the large deviation principle, in ${\mathcal X} $ defined in \eqref{defX}
 with the  rate function defined by \eqref{ratefc};
 see Theorem 4.4 in \cite{BD00} or
Theorem 5 in \cite{BD07}. This  concludes the proof of Theorem \ref{PGDunu}.
\hfill $\Box$
\vspace{.5cm}

 \textbf{Acknowledgment:} This work was partially written while  H. Bessaih
was invited professor at the University of Paris 1. The work of this author has also
been supported by the NSF grant No. DMS 0608494.

\end{document}